\documentclass[11pt,a4 paper,twoside, reqno]{amsart}
\usepackage[utf8]{inputenc}
\usepackage[OT1]{fontenc}
\usepackage[english]{babel}
\usepackage{dsfont}
\usepackage{amsfonts}
\usepackage{amsmath}
\usepackage{bbm}
\usepackage{wasysym}
\usepackage{amsthm}
\usepackage{amssymb}
\usepackage{float}
\usepackage{textcomp}
\usepackage{xcolor}
\usepackage{graphicx}
\usepackage{fancybox}
\usepackage{fancyhdr}
\usepackage{mathrsfs}
\usepackage{tikz}
\usetikzlibrary{arrows}
\usetikzlibrary{calc}
\usepackage{centernot}
\usepackage{listings}
\usepackage{faktor}
\usepackage{bm}
\usepackage{setspace}
\usepackage{hyperref}
\usepackage{newlfont}
\usepackage{geometry}
\usepackage{ccaption}
\usepackage{tipa}
\usepackage{units}
\usepackage[autostyle,italian=guillemets]{csquotes}
\usepackage{comment}
\usepackage{xfrac}

\usepackage{guit}
\usepackage{tikz-cd}
\usepackage{enumerate}

\everymath{\displaystyle}

\theoremstyle{definition}
\newtheorem{Def}{Definition}[section]

\newtheorem{es}[Def]{Example}
\newtheorem{ese}[Def]{Examples}
\newtheorem{as}[Def]{Assumption}

\newtheorem*{def:wsc}{Definition~\ref{wsc}}
\theoremstyle{remark}
\newtheorem{obs}[Def]{Remark}
\newtheorem{nota}[Def]{Notation}
\theoremstyle{plain}
\newtheorem{prop}[Def]{Proposition}

\newtheorem{lema}[Def]{Lemma}

\newtheorem{cor}[Def]{Corollary}

\newtheorem{teo}[Def]{Theorem}

\newtheorem*{teo:1stchar}{Theorem~\ref{1stchar}}
\newtheorem*{teo:bieq}{Theorem~\ref{bieq}}
\newtheorem*{teo:sharply}{Theorem~\ref{sharply}}
\newtheorem*{teo:raising}{Theorem~\ref{rainsing}}
\newtheorem*{teo:char-single}{Theorem~(part of \ref{char-single})}

\newcommand{\bb}{\mathbbm}
\newcommand{\bo}{\mathbf}
\newcommand{\A}{{\mathcal A}}
\newcommand{\B}{{\mathcal B}}
\newcommand{\C}{{\mathcal C}}
\newcommand{\D}{{\mathcal D}}
\newcommand{\E}{{\mathcal E}}

\newcommand{\G}{{\mathcal G}}

\newcommand{\I}{{\mathcal I}}
\newcommand{\J}{{\mathcal J}}
\newcommand{\K}{{\mathcal K}}
\renewcommand{\L}{{\mathcal L}}
\newcommand{\M}{{\mathcal M}}

\renewcommand{\P}{{\mathcal P}}
\newcommand{\Q}{{\mathcal Q}}

\renewcommand{\S}{{\mathcal S}}
\newcommand{\T}{{\mathcal T}}

\newcommand{\V}{{\mathcal V}}

\newcommand{\DD}{{\mathbb D}}
\newcommand\Set{\operatorname{\bo{Set}}}

\newcommand\Ab{\operatorname{\bo{Ab}}}

\newcommand{\LL}{{\mathbb L}}
\newcommand{\EE}{{\mathbb E}}
\newcommand\Str{\operatorname{\textnormal Str}}
\newcommand\Mod{\operatorname{\textnormal Mod}}

\newcommand{\tx}{\textnormal}

\newcommand{\op}{^\textnormal{op}}

\newcommand{\+}{\negthinspace^{^+}\negthinspace}

\newcommand*\cocolon{%
	\nobreak
	\mskip6mu plus1mu
	\mathpunct{}%
	\nonscript
	\mkern-\thinmuskip
	{:}%
	\mskip2mu
	\relax
}
\makeatother

\makeatletter
\newcommand{\changeoperator}[1]{%
	\csletcs{#1@saved}{#1@}%
	\csdef{#1@}{\changed@operator{#1}}%
}
\newcommand{\changed@operator}[1]{%
	\mathop{%
		\mathchoice{\textstyle\csuse{#1@saved}}
		{\csuse{#1@saved}}
		{\csuse{#1@saved}}
		{\csuse{#1@saved}}%
	}%
}
\makeatother

\makeatletter
\def\@tocline#1#2#3#4#5#6#7{\relax
	\ifnum #1>\c@tocdepth 
	\else
	\par \addpenalty\@secpenalty\addvspace{#2}%
	\begingroup \hyphenpenalty\@M
	\@ifempty{#4}{%
		\@tempdima\csname r@tocindent\number#1\endcsname\relax
	}{%
		\@tempdima#4\relax
	}%
	\parindent\z@ \leftskip#3\relax \advance\leftskip\@tempdima\relax
	\rightskip\@pnumwidth plus4em \parfillskip-\@pnumwidth
	#5\leavevmode\hskip-\@tempdima
	\ifcase #1
	\or\or \hskip 1em \or \hskip 2em \else \hskip 3em \fi%
	#6\nobreak\relax
	\hfill\hbox to\@pnumwidth{\@tocpagenum{#7}}\par
	\nobreak
	\endgroup
	\fi}
\makeatother

\changeoperator{sum}
\changeoperator{prod}
\changeoperator{coprod}

\title{More on soundness in the enriched context}
\author{Giacomo Tendas}
\address{Department of Mathematics, University of Manchester, Faculty of Science and Engineering, Alan Turing Building, M13 9PL Manchester, UK}
\email{giacomo.tendas@manchester.ac.uk}
\date{\today}
\thanks{The author acknowledges with gratitude the support of the EPSRC postdoctoral fellowship EP/X027139/1. The author also thanks Nathanael Arkor for valuable feedback}

\begin{document}
	
\begin{abstract}
	Working within enriched category theory, we further develop the use of soundness, introduced by Ad\'amek, Borceux, Lack, and Rosick\'y for ordinary categories. In particular we investigate: (1) the theory of locally $\Phi$-presentable $\V$-categories for a sound class $\Phi$, (2) the problem of whether every $\Phi$-accessible $\V$-category is $\Psi$-accessible, for given sound classes $\Phi\subseteq\Psi$, and (3) a notion of $\Phi$-ary equational theory whose $\V$-categories of models characterize algebras for $\Phi$-ary monads on $\V$.
\end{abstract}	
	
\maketitle

\setcounter{tocdepth}{1}
\tableofcontents

\section{Introduction}

The notion of sound class of categories (or doctrine) $\DD$ was introduced by Ad\'amek, Borceux, Lack, and Rosick\'y in~\cite{ABLR02:articolo}. One of the aims there was to introduce, for such a class $\DD$, the notion of $\DD$-accessible category, and prove that a category $\A$ is accessible (in the usual sense involving a regular cardinal) if and only if it is $\DD$-accessible for some sound class $\DD$. This way one eliminates the use of regular cardinals from the definition of accessibility. It is noted in the same paper, that actually a slightly weaker (and possibly less technical) notion of soundness can be used to obtain the same results. Following recent literature by Lack and the author~\cite{LT22:virtual,LT22:limits}, we call that weak soundness.

In short, a class of categories $\DD$ is {\em weakly sound} if every small and $\DD$-cocomplete category $\C$ is $\DD$-filtered; that is, if $\C$-colimits commute with $\DD$-limits in $\Set$.  If the class $\DD$ is pre-saturated (Definition~\ref{satu}), as it happens most of the times, this coincides with the notion of soundness introduced in \cite{ABLR02:articolo}. 

In the context of enriched categories, the concept was introduced in~\cite{DL07} and then in~\cite{LR11NotionsOL} to study generalized notions of local presentability. There, the class of categories $\DD$ is replaced by a class of weights $\Phi$, as weighted limits and colimits are the standard choice when dealing with enrichment. 

\begin{def:wsc}[\cite{ABLR02:articolo,LR11NotionsOL}]
	A locally small class of weights $\Phi$ is called {\em weakly sound} if any $\Phi$-continuous weight $M\colon\C\op\to\V$ (out of a small $\Phi$-cocomplete $\C$) is $\Phi$-flat; that is, if $M$-weighted colimits commute in $\V$ with $\Phi$-limits.
\end{def:wsc}

When $\Phi$ is the class of conical weights associated to the elements of some class of categories $\DD$, this corresponds to the notion mentioned earlier (Proposition~\ref{simpler}).

In this paper we will further develop the theory of (weak) soundness in the enriched context. The content can be subdivided into three main parts involving local presentability, accessibility, and equational presentations of enriched categories.

\vspace{7pt}
\noindent{\bf\underline{Part 1}.}
The first part (Section~\ref{soundness}) is devoted to the study of locally $\Phi$-presentable enriched categories, for $\Phi$ a weakly sound class. The first instance of such enriched notion appeared in Kelly's paper \cite{Kel82:articolo}, where he extended most of the ordinary results about locally presentable categories to the enriched context. These included the Gabriel-Ulmer duality, reflectivity and orthogonality theorems, as well as a characterization in terms of limit sketches. 

Here we deal with the more general notion of locally $\Phi$-presentable $\V$-category, which specializes to that of Kelly when $\Phi$ is the class of finite weighted limits. This was first considered in the ordinary setting in \cite{ABLR02:articolo} and \cite{centazzo2004characterization}, was partially extended to enriched categories in \cite{LR11NotionsOL}, and then again in a formal 2-categorical framework~\cite{di2023accessibility}. With our approach we do not need to assume the base of enrichment to be locally presentable (as in~\cite{Kel82:articolo}), but merely complete and cocomplete (as in \cite{Kel82:libro,LR11NotionsOL}); this allows us to capture some topological examples such as the cartesian closed category $\bo{CGTop}$ compactly generated topological spaces.

Given a weakly sound class $\Phi$, we say that a cocomplete $\V$-category $\K$ is {\em locally $\Phi$-presentable} if it has a strong generator $\G$ which is made of {\em $\Phi$-presentable} objects; that is, every $G\in\G$ is such that $\K(G,-)\colon\K\to\V$ preserves $\Phi$-flat colimits.  

In this context we can generalize most of the known results on locally presentable categories. Indeed, given a weakly sound class $\Phi$, we prove:

\begin{teo:1stchar}
	The following are equivalent for a $\V$-category $\K$: \begin{enumerate}\setlength\itemsep{0.25em}
		\item $\K$ is locally $\Phi$-presentable;
		\item $\K$ is $\Phi$-accessible and cocomplete;
		\item $\K$ is $\Phi$-accessible and complete;
		\item $\K\simeq\Phi\tx{-Cont}(\C,\V)$ for some small $\Phi$-complete $\C$;
		\item $\K$ is a $\Phi$-flat embedded and reflective subcategory of some $[\C,\V]$.
	\end{enumerate}
\end{teo:1stchar}

Above, a $\V$-category is called {\em $\Phi$-accessible} if it is the free cocompletion of a small $\V$-category under $\Phi$-flat colimits~\cite{ABLR02:articolo,LT22:virtual}. With the same hypotheses on $\V$ we can also prove a generalization of the Gabriel--Ulmer duality~\cite{GU71:libro,Kel82:articolo,centazzo2004characterization}:

\begin{teo:bieq}
	The following is a biequivalence of 2-categories:
	\begin{center}
		
		\begin{tikzpicture}[baseline=(current  bounding  box.south), scale=2]

			\node (f) at (0,0.4) {$\bo{L}\Phi\bo{P}(-,\V)\colon \bo{L}\Phi\bo{P}$};
			\node (g) at (2.6,0.4) {$\Phi\tx{-}\bo{Cont}\op\cocolon \Phi\tx{-}\bo{Cont}(-,\V)$};
			
			\path[font=\scriptsize]

			([yshift=1.3pt]f.east) edge [->] node [above] {} ([yshift=1.3pt]g.west)
			([yshift=-1.3pt]f.east) edge [<-] node [below] {} ([yshift=-1.3pt]g.west);
		\end{tikzpicture}
		
	\end{center}
\end{teo:bieq}

Here $\bo{L}\Phi\bo{P}$ is the 2-category of locally $\Phi$-presentable $\V$-categories, continuous and $\Phi$-flat colimits preserving $\V$-functors, and $\V$-natural transformations, while $\Phi\tx{-}\bo{Cont}$ is the 2-categories of those $\Phi$-complete $\V$-categories which are the Cauchy completions of small $\V$-categories, $\Phi$-continuous $\V$-functors, and $\V$-natural transformations. When $\V_0$ is locally presentable then the Cauchy completion of a small $\V$-category is still small~\cite{Joh1989:articolo}; therefore in that case the objects of $\Phi\tx{-}\bo{Cont}$ are just the small $\Phi$-complete and Cauchy complete $\V$-categories.

To prove additional results relating local $\Phi$-presentability with limit sketches and orthogonality classes, we shall need the base of enrichment $\V_0$ to be {\em locally bounded}; this includes all the locally presentable bases as well as those involving topological spaces (see Theorem~\ref{lphip-char}). In the case where $\V_0$ is also locally presentable, we prove that the 2-category $\bo{L}\Phi\bo{P}$ has all flexible limits which are computed as in $\V\tx{-}\bo{CAT}$ (Theorem~\ref{flexible}).

\vspace{7pt}
\noindent{\bf\underline{Part 2}.}
In the second part of the paper (Section~\ref{section-raising}), we drop the cocompleteness assumptions on our locally $\Phi$-presentable $\V$-categories, and focus just on $\Phi$-accessible enriched categories. We study the problem of {\em raising the index of accessibility}; that is, we study for which pairs of (weakly) sound classes of weights $\Phi\subseteq\Psi$ it is true that every $\Phi$-accessible $\V$-category is $\Psi$-accessible.

Ordinarily, and for $\Phi$ and $\Psi$ consisting respectively on the classes of $\alpha$-small and $\beta$-small diagrams (for regular cardinals $\alpha<\beta$), we know that the above condition holds if and only if $\alpha$ is {\em sharply less} than $\beta$ in the sense of~\cite[Definition~2.3.1]{MP89:libro}.  In our more general framework we introduce a notion of sharply less than relationship for sound classes of weights (which restricts to the original one when considering classes induced by regular cardinals) defined by the equivalent conditions below. This generalizes to the enriched settings parts of \cite[Theorem~2.3]{LT22:limits}.

\begin{teo:sharply}
	The following are equivalent for any given sound classes $\Phi\subseteq\Psi$: \begin{enumerate}\setlength\itemsep{0.25em}
		\item For any small $\V$-category $\C$ we have an equivalence $$ \Psi\+(\Phi^{^+}_\Psi\C) \simeq \Phi\+\C.$$
		\item If $M\colon\C\op\to\V$ is a $\Phi$-flat weight, then $\tx{Ran}_{J\op}M\colon(\Phi^{^+}_\Psi\C)\op\to\V$ is $\Psi$-flat, where $J\colon\C\to\Phi^{^+}_\Psi\C$ is the inclusion.
		\item Every $\Phi$-accessible $\V$-category $\A$ is $\Psi$-accessible and there is an equivalence  $$ \Phi^{^+}_\Psi\A_\Phi\simeq \A_\Psi.$$
		\item Every $\Phi$-accessible $\V$-category $\A$ is $\Psi$-accessible, and for any $\Phi\+$-cocontinuous $\V$-functor $F\colon\A\to \B$, between $\Phi$-accessible $\V$-categories, if $F(\A_\Phi)\subseteq\B_\Phi$ then also $F(\A_\Psi)\subseteq\B_\Psi$.
		\item For any small $\C$ the $\V$-category $\Phi\+\C$ is $\Psi$-accessible and the inclusion $J\colon \Phi\+\C\hookrightarrow[\C\op,\V]$ preserves the $\Psi$-presentable objects.
	\end{enumerate}
	When they hold, we say that {\em $\Phi$ is sharply less than $\Psi$}, and write $\Phi\trianglelefteq \Psi$.
\end{teo:sharply}

 Above, $\Phi^+$ is the class spanned by the $\Phi$-flat weights, and $\Phi^{^+}_\Psi$ denotes the class consisting of those weights in $\Psi$ which are $\Phi$-flat. Such a theorem is not only useful for relating $\Phi$- and $\Psi$-accessibility, but also to describe $\Phi$-flat colimits in simpler terms. Indeed, if $\Phi\triangleleft \Psi$, then a $\V$-category has $\Phi$-flat colimits if and only if it has $\Psi$-flat colimits and colimits weighted by $\Phi$-flat weights in $\Psi$.  
 
This new order relationship for sound classes of weights would not be useful without the following result. 
   
\begin{teo:raising}
	For any weakly sound class $\Phi$ there exists a regular cardinal $\alpha$ for which $$\Phi\triangleleft\P_\alpha.$$
\end{teo:raising}

In particular we obtain that for every sound class $\Phi$ there exists arbitrarily large sound classes which are sharply larger than $\Phi$ (Corollary~\ref{arbitr-large}).  Thus, for every weakly sound $\Phi$ there exists a regular cardinal $\alpha$ such that every $\Phi$-accessible $\V$-category $\A$ is also $\alpha$-accessible. This improves results from~\cite{BQR98,LT22:virtual} where the choice of $\alpha$ depends also on $\A$.
In Example~\ref{exe-sharp} we exhibit several pairs of sound classes where one is sharply less than the other.

\vspace{7pt}
\noindent{\bf\underline{Part 3}.}
The third and final part of the paper (Section~\ref{UAsection}) is devoted to the study of those enriched equational theories from~\cite{RT23EUA}, which are presented by {\em $\Phi$-ary equations} over a $\Phi$-ary language, for a given weakly sound class $\Phi$.  As for \cite{RT23EUA}, the point of this Section is to show that equations can be defined over a class of recursively generated terms, thus moving in a different direction with respect to earlier works \cite{LP,FH}.

A {\em language} in the sense of~\cite{LP,RT23EUA} is a set of function symbols with specified domain and codomain arities $(X,Y)$ varying among the objects of $\V$; for each such $\LL$ there is a corresponding notion of $\LL$-structure and of $\V$-category of $\LL$-structures. The class of {\em $\LL$-terms} is built recursively starting from the function symbols in $\LL$ and applying the rules of Definition~\ref{0-terms} (and is a subclass of the terms considered in \cite{LP}). Then, a theory $\EE$ over $\LL$ is a set of equations $(s=t)$ involving terms of the same arity; a model of such a theory is an $\LL$-structure which satisfies these equations, once interpreted. All this background material will be covered in Section~\ref{BackEUA}.

Starting from this, we define a language $\LL$ and an $\LL$-theory $\EE$ to be {\em $\Phi$-ary} if they involve function symbols and terms whose arities are (certain) $\Phi$-presentable objects (see Definition~\ref{terms} for the details). Among others, $\Phi$-theories characterize $\V$-categories of algebras of those enriched monads on $\V$ which preserve $\Phi$-flat colimits:

\begin{teo:char-single}
	The following are equivalent for a $\V$-category $\K$: \begin{enumerate}\setlength\itemsep{0.07em}
		\item $\K\simeq\Mod(\mathbb E)$ for a $\Phi$-ary equational theory $\mathbb E$ with recursively generated terms;
		\item $\K\simeq\tx{Alg}(T)$ for a monad $T$ on $\V$ preserving $\Phi$-flat colimits;
		\item $\K$ is cocomplete and has a $\Phi$-presentable and $\V$-projective strong generator $G\in\K$;
		\item $\K\simeq\Phi\I\tx{-Pw}(\T,\V)$ is equivalent to the $\V$-category of $\V$-functors preserving $\Phi\I$-powers, for some $\Phi\I$-theory $\T$.
	\end{enumerate}
	In particular $\K$ is locally $\Phi$-presentable.
\end{teo:char-single}

In (4) above, $\Phi\I$ is the closure of the unit $I$ in $\V$ under $\Phi$-colimits. When $\V$ is locally $\lambda$-presentable as a closed category and $\Phi$ is the class of the $\lambda$-small weights, this specializes to \cite[Theorem~5.14]{RT23EUA}. In Section~\ref{fp} we take $\V$ to be cartesian closed and $\Phi$ to consist of the class of finite products; in this case $\Phi$-flat colimits are called $\V$-sifted. The theorem above then provides a characterization of the $\V$-categories of algebras of strongly finitary (enriched) monads on $\V$; these have been studied for instance in~\cite{ADV2023sifted,KL93FfKe:articolo,ADVCALCO.2023.10,Par23}. In Appendix~\ref{Cat-sifted}, following~\cite{Bou10:PhD}, we characterize 2-dimensional sifted colimits.

\section{Background}

We fix as base of enrichment a symmetric monoidal closed category $\V=(\V_0,\otimes,I)$ which is complete and cocomplete, and we follow the notation of Kelly~\cite{Kel82:libro}
for matters about enrichment over $\V$.

The notion of free cocompletion under a class of weights will be central in the following sections. For this reason we recall some of the main notions from~\cite{KS05:articolo} below.

Given a class of weights $\Phi$ and a $\V$-category $\C$ we denote by $\Phi\C$ the {\em free cocompletion of $\C$ under $\Phi$-colimits}; this satisfies the following universal property: for any $\Phi$-cocomplete $\K$, precomposition by $Z$ induces an equivalence 
$$ \Phi\tx{-Coct}(\Phi\C,\K)\simeq\V\tx{-}\bo{CAT}(\C,\K); $$
here $\V\tx{-}\bo{CAT}(\C,\K)$ is the category of all $\V$-functors from $\C\to\K$ and $\V$-natural transformations between them, while we denote by $\Phi\tx{-Coct}(\Phi\C,\K)$ the full subcategory of $\V\tx{-}\bo{CAT}(\Phi\C,\K)$ spanned by the $\Phi$-cocontinuous $\V$-functors. The inverse of the equivalence above is given by left Kan extending along $Z$ (see \cite[Proposition~3.6]{KS05:articolo}).

Again by the results of~\cite{KS05:articolo}, the $\V$-category $\Phi\C$ can be identified with the closure of the representables in $[\C\op,\V]$ under $\Phi$-colimits. Note that $[\C\op,\V]$ may not itself be a $\V$-category (but just a $\V'$-category for some enlargement $\V'$ of $\V$) if $\C$ is not small; nonetheless $\Phi\C$ will always be a $\V$-category.

Of particular importance are the free cocompletion under the class $\P$ of all weights and the class $\Q$ of the Cauchy weights (those whose colimits commute with all limits in $\V$). 

Following~\cite{DL07}, given a $\V$-category $\C$, an element $F\in\P\C$ is a $\V$-functor $F\colon\C\op\to\V$ that is a (small) colimit of representables. Equivalently, $F\in\P\C$ if and only if it can be expressed as $F\cong\tx{Lan}_HFH$ where $H\colon\C'\hookrightarrow\C$ is a small full subcategory of $\C$. We refer to the objects of $\P\C$ as the {\em small presheaves} on $\C$.

\begin{Def}
	A class of weights $\Phi$ is called locally small if for any small $\V$-category $\C$ the free cocompletion $\Phi\C$ is also small. 
\end{Def}

If $\Phi$ is a small set of weights, then $\Phi$ is locally small. The class $\Q$ of the Cauchy weights is locally small whenever $\V$ is locally presentable  (\cite{Joh1989:articolo} or Corollary~\ref{Cauchy-small}). The class $\P$ of all small weights is not locally small in general.

\begin{Def}\label{satu}
	Given a class of weights $\Phi$, we denote by $\Phi^*$ the {\em saturation of $\Phi$}; this is defined by 
	$$ \Phi^*\cap [\C\op,\V]:=\Phi(\C)$$
	for any small $\V$-category $\C$. We say that $\Phi$ is {\em saturated} if $\Phi=\Phi^*$.
	The class $\Phi$ is called {\em pre-saturated} if for any $\V$-category $\C$ every object of the free cocompletion $\Phi\C$ is a $\Phi$-colimit of objects from $\C$.
\end{Def}

Every saturated class is pre-saturated by \cite{AK1988:articolo}. Pre-saturated classes have been studied for instance in~\cite[Appendix~A.1]{Ten22:phd}. Note that most commonly used classes of weights are pre-saturated, but in general not saturated. Indeed, the class of finite limits is pre-saturated but not saturated; its saturation is the class of {\em L-finite} limits (see~\cite[Section~3]{pare1990simply}).

\begin{Def}
	Let $\Phi$ be a locally small class of weights. A weight $M\colon\C\op\to\V$ is called {\em $\Phi$-flat} if $M$-weighted colimits commute in $\V$ with $\Phi$-limits; that is, if $M*-\colon[\C,\V]\to\V$ is $\Phi$-continuous. We denote by $\Phi^+$ the class of all $\Phi$-flat weights.
\end{Def}

Let $Y\colon\C\op\to[\C,\V]$; then, for a weight $M$ as above, we have $\tx{Lan}_YM\cong M*-$. It follows that $M$ is $\Phi$-flat if and only if $\tx{Lan}_YM$ is $\Phi$-continuous. When talking about $\Phi$-flat colimits we mean colimits weighted by $\Phi$-flat weights.

We recall the notion of soundness below:

\begin{Def}[\cite{ABLR02:articolo,LR11NotionsOL}]\label{wsc}
	A locally small class of weights $\Phi$ is called {\em weakly sound} if every $\Phi$-continuous weight $M\colon\C\op\to\V$, over a small $\Phi$-cocomplete $\C$, is $\Phi$-flat. 
\end{Def}

\begin{obs}\label{sat-sound}
	A locally small class $\Phi$ is instead called {\em sound}~\cite[Definition~3.2]{LT22:virtual} if for any weight $M$, whenever $M*-$ preserves $\Phi$-limits of representables, then $M$ is $\Phi$-flat. This corresponds to the notion of soundness originally introduced in \cite{ABLR02:articolo}, but is somewhat harder to deal with. \\
	Clearly, every sound class is weakly sound; the converse holds when $\Phi$ is pre-saturated by \cite[Proposition~3.4]{LT22:virtual}. (When $\Phi$ is actually saturated, the fact that weakly sound implies sound was proved in \cite[Proposition~3.8]{DV14elementary}.)
\end{obs}

Some examples:

\begin{es}\label{examplewsc}
	See \cite[Example~4.8]{LT22:limits} for details. In cases where
	$\V=\bo{Set}$ and we speak of a class of categories rather than a class of weights, we are referring to the corresponding conical limits.
	\begin{enumerate}
		\item $\Phi=\emptyset$. Then $\Phi^+=\P$ is the class of all weights. 
		\item $\V$ locally $\alpha$-presentable as a closed category, $\Phi$ is the class of $\alpha$-small weights defined by Kelly~\cite{Kel82:articolo}; this is generated by the $\alpha$-small conical limits and powers by $\alpha$-presentable objects. 
		\item $\V$ a symmetric monoidal closed finitary quasivariety as in \cite{LT20:articolo}, $\Phi$ is the class for finite products and finitely presentable projective powers. Examples of such $\V$ are the categories $\bo{Ab}$ of abelian groups, $\bo{GAb}$ of graded abelian groups, and $\bo{DGAb}$ of differentially graded abelian groups.
		\item $\V$ cartesian closed, $\Phi$ is the class for finite products. Examples of such $\V$ are the categories $\bo{Pos}$ of partially ordered sets, $\bo{Cat}$ of small categories, and $\bo{SSet}$ of simplicial sets; but also $\V=\bo{CGTop},\bo{HCGTop}$, and $\bo{QTop}$ of compactly generated topological spaces, Hausdorff such, and quasi-topological spaces, which are (locally bounded but) not locally presentable. (\cite[Lemma~2.3]{KL93FfKe:articolo}). 
		\item $\V=\bo{Set}$, $\Phi=\{\emptyset\}$. 
		\item $\V=\bo{Set}$, $\Phi$ consists of the finite connected categories. Then $\Psi^+$ is generated by coproducts of filtered categories.
		\item $\V=\bo{Set}$, $\Phi$ is the class of finite non empty categories. Then $\Psi^+$ is generated by the filtered categories plus the empty category.
		\item $\V=\bo{Set}$, $\Phi$ is the class of finite discrete non empty categories. Then $\Psi^+$ is generated by the sifted categories plus the empty category.
	\end{enumerate}
	See also Example~\ref{inducedfromSet} of which (2) and (3) are a particular case.
\end{es}

We can characterize soundness as follows, somewhat expanding the equivalent conditions given by~\cite[Propositions~4.5 \&~4.9]{LT22:limits}.

\begin{prop}
	The following are equivalent for $\Phi$:\begin{enumerate}
		\item $\Phi$ is weakly sound;
		\item for any weight $M\colon \C\op\to\V$ the $\V$-functor $\tx{Ran}_{J\op}M$ is $\Phi$-flat, where $J\colon\C\hookrightarrow\Phi\C$ is the inclusion.
	\end{enumerate}
\end{prop}
\begin{proof}
	$(1)\Rightarrow(2)$. This true by definition of weak soundness since $\tx{Ran}_{J\op}M$ is $\Phi$-continuous by construction.
	
	$(2)\Rightarrow(1)$. We have for any $M\colon \C\op\to\V$ that
	$$ M\cong (\tx{Ran}_{J\op}M)J\cong \tx{Ran}_{J\op}M* W  $$
	with $W\colon \Phi\C\hookrightarrow[\C\op,\V]$. Since $\Phi\C\subseteq [\C\op,\V]_\Phi$ (essentially by definition) and $\tx{Ran}_{J\op}M$ is $\Phi$-flat, one concludes thanks to \cite[4.9]{LT22:limits} which says that $\Phi$ is weakly sound if and only if every object of $[\C\op,\V]$ is a $\Phi$-flat colimit of objects from $\Phi\C$. 
\end{proof}

\section{Soundness and local presentability}\label{soundness}

In this section we treat local presentability with respect to a weakly sound class of weight. Some of the results have either been proved in \cite{LR11NotionsOL} and \cite{di2023accessibility} (we shall say when that is the case), or are a standard generalization of facts in \cite{ABLR02:articolo} and \cite{Kel82:articolo}. We collect all these and new results here to provide a comprehensive treatment of the theory of local presentability in this more general setting. These results will be relevant in Section~\ref{UAsection}.

In Section~\ref{basic-presentable} below we introduce the basic notions, while in Section~\ref{main-presentable} we prove the characterization theorems for locally $\Phi$-presentable $\V$-categories. Then in Section~\ref{sound-set} we consider the case when the enriched class $\Phi$ is induced (in a precise way) by an ordinary sound class of weights. This allows us to prove, in a new way, Johnson's theorem about the smallness of Cauchy completions for categories enriched over a locally presentable base \cite{Joh1989:articolo}.

\subsection{Definitions and basic results}\label{basic-presentable}

For this section we fix $\V=(\V_0,\otimes,I)$ to be symmetric monoidal closed, complete, and cocomplete, and consider $\Phi$ a locally small class of weights (later this will be assumed to be weakly sound).

\begin{Def}
	Let $\K$ be a $\V$-category with $\Phi$-flat colimits; we say that $X\in\K$ is $\Phi$-presentable if $\K(X,-)\colon\K\to\V$ preserves $\Phi$-flat colimits. Denote by $\K_{\Phi}$ the full subcategory of $\K$ made of the $\Phi$-presentable objects.
\end{Def}

The following was first considered in the enriched context in \cite[Section~6.3]{LR11NotionsOL} with an equivalent definition. They say that a $\V$-category is locally $\Phi$-presentable if it is $\Phi$-accessible (Definition~\ref{Phi-acc}) and cocomplete, while we follow the approach that involves a strong generator. That these provide equivalent notions will follow from our Theorem~\ref{1stchar}.

\begin{Def}
	Let $\K$ be a $\V$-category. We say that $\K$ is locally $\Phi$-presentable if it is cocomplete and it has a small (enriched) strong generator $J\colon \G\hookrightarrow\K$ made of $\Phi$-presentable objects; this means that the induced singular $\V$-functor $\K(J,1)\colon\K\to[\G\op,\V]$ is conservative.
\end{Def}

It follows that, when $\V=\bo{Set}$ and $\Phi$ is the class for $\lambda$-small limits, we recover Gabriel and Ulmer's notion of locally $\lambda$-presentable category~\cite{GU71:libro}. When $\V$ is locally $\lambda$-presentable as a closed category and $\Phi$ is the class for all $\lambda$-small weighted limits, then we recover Kelly's definition from~\cite{Kel82:articolo}.

\begin{prop}\label{phi-pres-presheaves}
	The $\V$-category $\V$ itself is locally $\Phi$-presentable, and so is $[\C,\V]$ for any small $\C$.
\end{prop}
\begin{proof} 
	The unit is a strong generator of $\V$ and is $\Phi$-presentable. Similarly, the representables form a small strong generator of $[\C,\V]$ made of $\Phi$-presentable objects.
\end{proof}

\begin{prop}\label{conservative}
	If $\K$ is locally $\Phi$-presentable, $\L$ is cocomplete, and $F\colon\L\to\K$ is a conservative right adjoint that preserves $\Phi$-flat colimits, then $\L$ is also locally $\Phi$-presentable.
\end{prop}
\begin{proof}
	Let $\G$ be a strong generator of $\K$ made of $\Phi$-presentable objects, and $L$ be a left adjoint to $F$. Then $L(\G)$ is a strong generator of $\L$ (since $F$ is conservative) and is made of $\Phi$-presentable objects: if $G\in\G$ then
		$$ \L(FG,-)\cong\K(G,F-) $$
	preserves $\Phi$-flat colimits since $F$ does and $G\in\K_\Phi$.
\end{proof}

The notion of $\Phi$-accessible $\V$-category has been already introduced in \cite[Section~3]{LT22:virtual} for a locally presentable base of enrichment. They can be defined exactly in the same way also in this context (see below), and most of the results of \cite{LT22:virtual} extend with the only change that Cauchy completions may not be small. 

\begin{Def}[\cite{LT22:virtual}]\label{Phi-acc}
	We say that $\A$ is {\em $\Phi$-accessible} if it has $\Phi$-flat colimits and there exists a small $\C\subseteq\A_\Phi$ such that every object of $\A$ can be written as a $\Phi$-flat colimit of objects from $\C$.
\end{Def}

Equivalently, $\K$ is $\Phi$-accessible if and only if it is the free cocompletion of a small $\V$-category under $\Phi$-flat colimits (\cite{LT22:virtual}). In this section we are only interested in the complete or cocomplete case (which will then imply local $\Phi$-presentability). 

For the notions of local $\Phi$-presentability and $\Phi$-accessibility to interact nicely, we need $\Phi$ to be a weakly sound.

\subsection{Characterization theorems}\label{main-presentable}

\begin{as}
	From now on we fix a locally small and weakly sound class $\Phi$. 
\end{as}

The following proposition characterizes the $\Phi$-presentable objects of a locally $\Phi$-presentable $\V$-category, and shows that they freely generate under $\Phi$-flat colimits.

\begin{prop}\label{K_Phi}
	Let $\K$ be locally $\Phi$-presentable with strong generator $\G\subseteq\K_\Phi$, and let $H\colon\K_{\Phi}\hookrightarrow\K$ be the inclusion. Then:\begin{enumerate}\setlength\itemsep{0.25em}
		\item $\K_{\Phi}$ is closed in $\K$ under $\Phi$-colimits;
		
		\item the closure of $\G$ under $\Phi$-colimits and Cauchy colimits is $\K_\Phi$;
		
		\item if $\C$ is the closure of $\G$ under $\Phi$-colimits, then $\C$ is small, dense, its Cauchy completion is $\K_\Phi$, and every $X\in\K$ can be expressed as the $\Phi$-flat colimit:
		$$ X\cong \K(H_\C-,X)*H_\C, $$
		where $H_\C\colon\C\to\K$ is the inclusion.
		
		\item every $X\in\K$ can be expressed as the (possibly large) $\Phi$-flat colimit:
		$$ X\cong \K(H-,X)*H. $$
	\end{enumerate}
\end{prop}
\begin{proof}
	The points (1, 2, 4) are proved as in~\cite[Proposition~3.9]{LT22:virtual} with the only difference that one should keep in mind that the Cauchy completion $\Q(\C)$ of a small $\V$-category $\C$ may not be small, and that for any $X\in\V$ the restriction $\K(J-,X)\colon \C\op\to \V$ is $\Phi$-flat (the inclusion $J\colon\C\to \K$ is $\Phi$-cocontinuous).
	
	(3). We shall prove first that $\C$ is dense. Consider then the conservative functor $Z\colon\K\to[\C\op,\V]$ such that $ZX=\K(H-,X)$, where $H\colon\C\to\K$ is the inclusion; note also that $Z$ preserves $\Phi$-flat colimits and $ZH\cong Y\colon\C\to[\C\op,\V]$. Given $X\in\K$ the functor $ZX$ is $\Phi$-continuous and hence $\Phi$-flat, our aim is to prove that the comparison $ZX*H\to X$ in $\K$ is an isomorphism. But for that is enough to prove that its image through $Z$ is an isomorphism, and in fact we have
	$$ Z(ZX*H)\cong ZX*ZH \cong ZX*Y \cong ZX.$$%
	Thus $ZX*H\cong X$ and $\C$ is dense.
	
	Density then implies that $ X\cong \K(H_\C-,X)*H_\C $ for any $X\in\K$; this colimit is $\Phi$-flat since $\K(H_\C-,X)$ is $\Phi$-continuous. The fact that $\K_\Phi$ is the Cauchy completion of $\C$ follows from the same arguments as \cite[Proposition~3.9]{LT22:virtual}.
\end{proof}

Next we characterize those $\V$-functors that preserve all limits and $\Phi$-flat colimits. In the statement we assume $\Phi$-accessibility and completeness of the domain $\V$-category; it will follow from Theorem~\ref{1stchar} that this is equivalent to being locally $\Phi$-presentable.

\begin{prop}\label{lpAFT}
	Let $\K$ be $\Phi$-accessible and complete, and $\A$ be $\Phi$-accessible. Then the following are equivalent for $F\colon \K\to \A$:\begin{enumerate}\setlength\itemsep{0.25em}
		\item $F$ is continuous and preserves $\Phi$-flat colimits;
		\item $F$ preserves $\Phi$-flat colimits and has a left adjoint;
		\item $F$ has a left adjoint which preserves the $\Phi$-presentable objects.
	\end{enumerate}
\end{prop}
\begin{proof}
	$(1)\Rightarrow (2)$. Let $F\colon \K\to\A$ be continuous and $\Phi$-flat colimit preserving. We need to prove that for each $A\in\A$ the $\V$-functor $\A(A,F-)$ is representable. Since $\A$ is $\Phi$-accessible it is enough to consider only the $\Phi$-presentable $A\in\A$ (then the left adjoint extends under $\Phi$-flat colimits to all of $\A$). Consider then $A\in\A_\Phi$; by \cite[Theorem~4.80]{Kel82:libro} to prove that $\A(A,F-)$ is representable is the same as proving that $\{\A(A,F-),\tx{id}_{\K}\}$ exists and is preserved by $\A(A,F-)$. By hypothesis $\A(A,F-)$ preserves $\Phi$-flat colimits and $\K$ is $\Phi$-accessible; thus $\A(A,F-)$ is the left Kan extension of its restriction to the $\Phi$-presentable objects. Let $H\colon \K_{\Phi}\to\K$ be the inclusion; then this means that $\A(A,F-)\cong \A(A,FH-)*YH\op$, where $Y\colon \K\op\to[\K,\V]$ is Yoneda. We can consider then the following chain of isomorphisms, each side existing if the other does:
	\begin{equation*}
		\begin{split}
			\{\A(A,F-),\tx{id}_{\K}\}&\cong \{\A(A,FH-)*YH\op,\tx{id}_{\K}\}\\
			&\cong \{\A(A,FH-),\{YH\op,\tx{id}_{\K}\}\}\\
			&\cong \{\A(A,FH-),H\op\}\\
		\end{split}
	\end{equation*}
	and the last limit exists since $\A$ is complete. Moreover this is preserved by $\A(A,F-)$ being continuous.
	
	$(2)\Rightarrow(3)$. Let $L$ be the left adjoint to $F$ and $A\in\A_{\Phi}$; then $\K(LA,-)\cong \A(A,F-)$ preserves $\Phi$-flat colimits since $F$ and $\A(A,-)$ do.
	
	$(3)\Rightarrow(1)$. $F$ is continuous since it is a right adjoint. Moreover, the set of $\Phi$-presentable objects of $\A$ forms a strong generator, and homming out of them preserves $\Phi$-flat colimits; thus $F$ preserves them if and only if each $\A(A,F-)$, $A\in\A_{\Phi}$ preserves them. But $\A(A,F-)\cong\K(LA,-)$, so it preserves $\Phi$-flat colimits since $L$ preserves the $\Phi$-presentable objects.
\end{proof}

The equivalence $(2)\Leftrightarrow(4)$ below is given by \cite[6.4]{LR11NotionsOL}. While, when $\V_0$ is locally presentable, the first three equivalences and the fifth follow from \cite[4.22]{LT22:limits}.

\begin{teo}\label{1stchar}
	The following are equivalent for a $\V$-category $\K$: \begin{enumerate}\setlength\itemsep{0.25em}
		\item $\K$ is locally $\Phi$-presentable;
		\item $\K$ is $\Phi$-accessible and cocomplete;
		\item $\K$ is $\Phi$-accessible and complete;
		\item $\K\simeq\Phi\tx{-Cont}(\C,\V)$ for some small $\Phi$-complete $\C$;
		\item $\K$ is a $\Phi$-flat embedded and reflective subcategory of some $[\C,\V]$.
	\end{enumerate}
	Moreover in $(4)$ we can choose $\C$ to be $\K_\Phi\op$ (even if that may not be small).
\end{teo}
\begin{proof}
	$(1)\Rightarrow(2)$. Let $\K$ be locally $\Phi$-presentable and $\G$ be a small strong generator made of $\Phi$-presentable objects. Let $\C$ be the closure of $\G$ under $\Phi$-colimits; this is dense and small by Proposition~\ref{K_Phi}. Now, by density of $\C$, every $X$ in $\K$ can be written as the colimit 
	$$ X\cong \K(H-,X)*H, $$
	where $\K(H-,X)$ is $\Phi$-flat (being $\Phi$-continuous). Thus (2) holds.
	
	$(2)\Rightarrow(1)$ follows by definition, since a dense generator is in particular a strong generator.
	
	$(1)\Rightarrow(4)$. Consider $\C$ as in Proposition~\ref{K_Phi}(3) and let $H=H_\C$. By density we have a fully faithful $J\colon\K\hookrightarrow[\C,\V]$ sending $X$ to $JX:=\K(H-,X)$, where $H\colon\K_{\Phi}\hookrightarrow\K$ is the inclusion. Note that $J$ preserves $\Phi$-flat colimits since we are homming out of $\Phi$-presentable objects. Moreover, since $H$ preserves $\Phi$-colimits, every $\V$-functor $\K(H-,X)$ is $\Phi$-continuous, so that $J$ restricts to $\K\hookrightarrow\Phi\tx{-Cont}(\C,\V)$. Finally, any $F\in \Phi\tx{-Cont}(\C,\V)$ is $\Phi$-flat since $\Phi$ is weakly sound; therefore we can consider the $\Phi$-flat colimit $F*H\in K$. This is preserved by $J$ so that $J(F*H)\cong F*(JH)\cong F*Y_\C \cong F$. It follows that $\K$ is equivalent to $\Phi\tx{-Cont}(\C,\V)$.
	
	$(4)\Rightarrow(3)$. It is clear that $\K\simeq\Phi\tx{-Cont}(\C,\V)$ is complete since $\Phi$-continuous $\V$-functors are stable under limits. To show that it is $\Phi$-accessible notice that it has $\Phi$-flat colimits (since these commute in $\V$ with $\Phi$-limits), that the inclusion $H\colon \C\op\hookrightarrow\K$ is dense, and that $\C$ is made of $\Phi$-presentable objects. Moreover, for any $X\in\K$ we can write $X\cong X*H$ (by density of $H$) where $X$ is $\Phi$-flat (being $\Phi$-continuous). Thus $\K$ is $\Phi$-accessible.
	
	$(3)\Rightarrow(5)$. Consider the inclusion $J\colon \K\hookrightarrow[\G\op,\V]$ induced by a dense $\G$ made of $\Phi$-presentable objects, which is continuous and $\Phi$-flat colimit preserving. It follows then by the previous proposition that $J$ has a left adjoint. 
	
	$(5)\Rightarrow(1)$. By Proposition \ref{phi-pres-presheaves} the $\V$-category $[\C,\V]$ is locally $\Phi$-presentable. Let $J$ be the inclusion and $L$ the left adjoint; then $\K$ is cocomplete being a reflective subcategory of a cocomplete one. Let $\G:=L([\C,\V]_\Phi)$; since $\K(LX,-)\cong[\C,V](X,J-)$ and $J$ preserves $\Phi$-flat colimits, it follows that every element of $\G$ is $\Phi$-presentable; moreover $\G$ generates $\K$ under $\Phi$-flat colimits: for each $K\in\K$ the object $JK\cong M*H$ is a $\Phi$-flat colimit of $\Phi$-presentable objects of $[\C,\V]$; therefore $K\cong LJK\cong L(M*H)\cong M*LH$ is a $\Phi$-flat colimit of objects from $\G$. This means exactly that $\K$ is locally $\Phi$-presentable.
\end{proof}

\begin{cor}\label{Phi-flat-preserve}
	Let $\K$ be a locally $\Phi$-presentable $\V$-category. For any $\V$-category $\L$ with $\Phi$-flat colimits, left Kan extending along the inclusion $H\colon\K_\Phi\hookrightarrow\K$ induces an equivalence
	$$ \tx{Lan}_H\colon [\K_\Phi,\L]\xrightarrow{\ \simeq \ } \Phi\+\tx{-Coct}(\K,\L), $$
	where $\Phi\+\tx{-Coct}(\K,\L)$ is the $\V$-category of $\V$-functors $\K\to\L$ preserving $\Phi$-flat colimits.
\end{cor}
\begin{proof}
	By the proposition above and the equivalent formulations of a $\Phi$-accessibility, the $\V$-category $\K$ is the free cocompletion of of $\K_\Phi$ under $\Phi$-flat colimits. Thus the equivalence above is just expressing its universal property. 
\end{proof}

\begin{es}
	Consider a cartesian closed base of enrichment $\V$ together with the sound class $\bo{Fp}$ for finite products (Example~\ref{examplewsc}). 
	Given an ordinary category $\C$ with finite products, it is easy to see that the free $\V$-category $\C_\V$ over $\C$ still has finite products and that these are computed as in $\C$. It follows that the $\V$-category $\tx{Fp}(\C_\V,\V)$ is locally $\tx{Fp}$-presentable (by Theorem~\ref{1stchar}), so that all the results of this section apply. But if we take its underlying category, by construction we have that 
	$$ \tx{Fp}(\C_\V,\V)_0\cong \tx{Fp}(\C,\V_0) $$
	which is in general not locally $\bo{Fp}$-presentable (that is, not a finitary variety), nor locally presentable unless $\V_0$ is.
	
	For instance, we can consider $\V=\bo{CGTop}$, the cartesian closed category of compactly generated topological spaces (which is not a variety and not locally presentable), and the ordinary Lawvere theory for groups. It follows from the arguments above, that the enriched category of groups internal to $\bo{CGTop}$ is locally $\bo{Fp}$-presentable, but its underlying ordinary category is not.
\end{es}

We now prove an enriched analogue of the Gabriel--Ulmer duality for sound classes that will generalize the ordinary version of~\cite{centazzo2004characterization} (see also \cite[Remark~3.8]{tendas2023dualities}).
Denote by $\bo{L}\Phi\bo{P}$ the 2-category of locally $\Phi$-presentable $\V$-categories, continuous and $\Phi$-flat colimit preserving $\V$-functors, and $\V$-natural transformations. On the other hand, let $\Phi\tx{-}\bo{Cont}$ be the 2-category of $\Phi$-complete $\V$-categories which are the Cauchy completion of a small $\V$-category, $\Phi$-continuous $\V$-functors, and $\V$-natural transformation.

\begin{nota}\label{notation}
	Given $\K\in \bo{L}\Phi\bo{P}$, we denote with $\bo{L}\Phi\bo{P}(\K,\V)$ --- bolded --- the hom-category in $\bo{L}\Phi\bo{P}$. This is the (ordinary) category of continuous and $\Phi$-flat colimit preserving $\V$-functors $\K\to\V$, and $\V$-natural transformations between them. Instead, we denote by $\tx{L}\Phi\tx{P}(\K,\V)$ --- unbolded --- the $\V$-category obtained by taking the full subcategory of $[\K,\V]$ spanned by the continuous and $\Phi$-flat colimit preserving $\V$-functors (the fact that this is actually a $\V$-category, and not just a $\V'$-category for some universe enlargement $\V'\supseteq\V$, follows from the theorem below). Therefore we obtain an equality
	$$\tx{L}\Phi\tx{P}(\K,\V)_0=\bo{L}\Phi\bo{P}(\K,\V).$$
	We use the same notational convention for $\C\in\Phi\tx{-}\bo{Cont}$, where the $\V$-category $\Phi\tx{-}\tx{Cont}(\C,\V)$ has already been defined. Therefore we obtain an equality
	$$ \Phi\tx{-}\tx{Cont}(\C,\V)_0=\Phi\tx{-}\bo{Cont}(\C,\V) $$
	of ordinary categories. 
\end{nota}

\begin{teo}\label{bieq}
	The following is a biequivalence of 2-categories:
	\begin{center}
		
		\begin{tikzpicture}[baseline=(current  bounding  box.south), scale=2]

			\node (f) at (0,0.4) {$\tx{L}\Phi\tx{P}(-,\V)\colon \bo{L}\Phi\bo{P}$};
			\node (g) at (2.6,0.4) {$\Phi\tx{-}\bo{Cont}\op\cocolon \Phi\tx{-}\tx{Cont}(-,\V)$};
			
			\path[font=\scriptsize]

			([yshift=1.3pt]f.east) edge [->] node [above] {} ([yshift=1.3pt]g.west)
			([yshift=-1.3pt]f.east) edge [<-] node [below] {} ([yshift=-1.3pt]g.west);
		\end{tikzpicture}
		
	\end{center}
\end{teo}
\begin{proof}
	Let $\K$ be locally $\Phi$-presentable; then we first note that $\tx{L}\Phi\tx{P}(\K,\V)\simeq \K_\Phi\op$, induced by the restriction of the Yoneda embedding $Y\colon \K\op\to[\K,\V]$. It is clear that for every $X\in\K_{\Phi}$ we have $\K(X,-)\in \tx{L}\Phi\tx{P}(\K,\V)$. Conversely let $F\colon \K\to\V$ be continuous and $\Phi$-flat colimits preserving; then by the previous proposition it has a left adjoint $L$. As a consequence $\K(LI,-)\cong[I,F-]\cong F$ and $LI\in\K_{\Phi}$, proving the equivalence. Thus the 2-functor is well defined since $\K_\Phi\op$ is $\Phi$-complete and the Cauchy completion of a small category.  Moreover, given $F\colon \K\to\A$ in $\bo{L}\Phi\bo{P}$ it follows that $\tx{L}\Phi\tx{P}(F-,\V)$ corresponds, up to the equivalence above, to the restriction of the left adjoint of $F$ to the $\Phi$-presentable objects. Now, the fact that $\Phi\tx{-}\tx{Cont}(\tx{L}\Phi\tx{P}(\K,\V) ,\V)\cong \K$ follows directly from $(1)\Rightarrow(2)$ of Theorem \ref{1stchar}.
	
	Let now $\C=\Q(\D)$ be $\Phi$-complete with $\D$ small; we can assume $\D$ to be $\Phi$-complete as well. Then $$ \K:=\Phi\tx{-}\tx{Cont}(\C,\V)\simeq \Phi\tx{-}\tx{Cont}(\D,\V)$$%
	is locally $\Phi$-presentable by Theorem \ref{1stchar}. By the arguments above, is then enough to prove that $\K_\Phi$ consists just of the representable functors, that is: $\K_\Phi\simeq\C\op$. On one hand every representable $YC=\C(C,-)$ in $\K$ is clearly $\Phi$-presentable since $\K(YC,-)\cong \tx{ev}_C\circ J$, where $J\colon \Phi\tx{-}\tx{Cont}(\C,\V)\to[\C,\V]$ is the inclusion, and both $\tx{ev}_C$ and $J$ preserves $\Phi$-flat colimits. On the other hand, let $X\in\K_\Phi$; then $[\C,\V](JX,J-)$ preserves $\Phi$-flat colimits. By \cite[Proposition~6.14]{KS05:articolo}, $JX$ is representable if and only if $[\C,\V](JX,-)$ preserves the colimit $JX*Y$ (since $\C$ is Cauchy complete). But $JX$ is $\Phi$-continuous, and hence $\Phi$-flat; thus $JX*Y\cong J(X*H)$ where $H$ is the inclusion of $\C$ in $\K$. Thus the thesis follows from the fact that $[\C,\V](JX,J-)$ preserves $\Phi$-flat colimits.
\end{proof}

\begin{obs}
	Local presentability has been considered in a formal 2-categorical context in~\cite{di2023accessibility}. Given a 2-category $\mathbb K$, the 2-categorical notion of local presentability relies on the choice of a {\em context} $\nu=(\S,\P)$ as in \cite[Section~2.4]{di2023accessibility}. In our case we take $\mathbb K=\V$-$\bo{CAT}$, the 2-category of large $\V$-categories, $\V$-functors, and $\V$-natural transformations; and we consider the context $\nu=(\Phi\+,\P)$ where $\Phi\+$ and $\P$ are the pseudomonads induced by freely adding $\Phi$-flat colimits and all small colimits respectively.\\
	Given this, to compare our notion with theirs, one need to understand the notion of {\em $\P$-petit} and {\em $\S$-Cauchy complete object} \cite[Definitions~2.37, 3.3]{di2023accessibility} (at least for $\S=\Phi\+$). Arguing as in \cite[Remark~3.4]{di2023accessibility}, it is easy to see that a $\V$-category $\C$ is Cauchy complete (in the standard enriched sense) if and only if it is $\Phi\+$-Cauchy complete. Similarly, one can also show that every $\V$-category $\C$ that is the Cauchy cocompletion of a small $\V$-category is $\S$-petit (for any $\S$ in a context $\nu$); however, we do not know if the converse holds. \\
	If the converse did hold, then the {\em $\nu$-presentable} objects (\cite[Definition~2.40]{di2023accessibility}) of $\V$-$\bo{CAT}$ would coincide with our locally $\Phi$-presentable $\V$-categories. In this case then the equivalence of (2),(3),(4), and (5) in Theorem~\ref{1stchar}, as well as the duality of Theorem~\ref{bieq} above, would follow from the results in Section~2 and 3 of \cite{di2023accessibility}.
\end{obs}

Next we need to recall the notion of limit sketches. While the ordinary notion is due to Ehresmann~\cite{Ehr68:sketches}, these were first considered in the enriched context by Kelly~\cite{Kel82:libro}, and have been used to classify locally presentable enriched categories in~\cite{Kel82:articolo}, as well as in the context of locally bounded enriched categories~\cite{LP21}. Mixed sketches (involving also colimits specifications) are used to characterize enriched accessible categories~\cite{BQR98,LT22:virtual,LT22:limits}.  

\begin{Def}[\cite{Kel82:libro}]
	A {\em limit sketch} is the data $\S=(\B,\mathbb{L})$ of a small $\V$-category $\B$ and a set $\LL$ of cylinders $M\Rightarrow \B(B,H-)$ in $\B$. Given a class of weights $\Phi$, if every weight $M $ appearing in $\mathbb{L}$ lies in $\Phi$, we call $\S$ a {\em $\Phi$-limit sketch}. 
\end{Def}

\begin{Def}
	Given a limit sketch $\S=(\B,\mathbb{L})$, a model of $\S$ is a $\V$-functor $F\colon\B\to\V$ which sends every cylinder in $\LL$ to a limit cylinder in $\V$.
\end{Def}

The notion of orthogonality below is very closely related to that of limit sketch (as we will see in Theorem~\ref{lphip-char}). These have been studied again in \cite{Kel82:libro,Kel82:articolo,LP21}.

\begin{Def}[\cite{Kel82:libro}]
	Let $\K$ be a $\V$-category and $f\colon A\to B$ an arrow in $\K$. We say that $X\in\K$ is orthogonal with respect to $f$ if $\K(X,f)$ is an isomorphism in $\V$. Let $\M$ be a set of arrows in $\K$; we denote by $\M^\perp$ the full subcategory of $\K$ spanned by the objects of $\K$ which are orthogonal with respect to every morphism in $\M$.  
\end{Def}

\begin{Def}	
	A full subcategory $\L$ of $\K$ is called an {\em orthogonality class} if it arises as $\M^\perp$ for some $\M$ as above. If all the arrows in $\M$ have $\Phi$-presentable domain and codomain, we say that $\L$ is a {\em $\Phi$-orthogonality class}. 
\end{Def}

The following theorem then characterizes locally $\Phi$-presentable $\V$-categories. For one implication we need the further assumption of $\V_0$ being locally bounded \cite[Section~6.1]{Kel82:libro}. We don't really need to recall the definition here; it will suffice to say that locally bounded categories include all locally presentable categories as well as certain categories of topological spaces (see \cite{LP21} for a good amount of examples).

\begin{teo}\label{lphip-char}
	Consider the following conditions for a $\V$-category $\K$:\begin{enumerate}\setlength\itemsep{0.25em}
		\item $\K$ is locally $\Phi$-presentable;
		\item $\K\simeq\tx{Mod}(\S,\V)$ for a $\Phi$-limit sketch $\S=(\C,\mathbb{L})$;
		\item $\K$ is equivalent to a $\Phi$-orthogonality class in some $[\C,\V]$;
	\end{enumerate}
	Then, $(1)\Rightarrow(2)\Rightarrow(3)$ always, and $(3)\Rightarrow(1)$ if $\V_0$ is locally bounded.
\end{teo}
\begin{proof}
	$(1)\Rightarrow (2)$. By Proposition~\ref{K_Phi} we can write $\K\simeq\Phi\tx{-Cont}(\C,\V)$ for some small $\Phi$-complete $\C$. Now, for each diagram $(M\colon\B\to\V,G\colon\B\to\C)$ with $M$ in $\Phi$ let 
	$$\lambda_{M,G}\colon M\to \C(\{M,G\},G-)$$%
	be its limiting cylinder. We consider the sketch $\S=(\C,\mathbb{L})$ with $\mathbb{L}$ being the collection of all these limiting cylinders; this is in $\Phi$ by construction. Then a $\V$-functor $F\colon \C\to\V$ is $\Phi$-continuous if and only if it sends each cylinder in $\mathbb{L}$ to a limiting cylinder in $\V$. It then follows at once that $\Phi\tx{-Cont}(\C,\V)=\tx{Mod}(\S,\V)$.
	
	$(2)\Rightarrow (3)$. Let $\S=(\C,\mathbb{L})$ be a $\Phi$-limit sketch. For each $\lambda\colon M\to \C(X,G-)$ in $\mathbb{L}$ consider the composite
	$$ \bar{\lambda}\colon M\stackrel{\lambda}{\longrightarrow}\C(X,G-)= \C\op(G\op-,X)\stackrel{Y}{\longrightarrow}[\C,\V](YG\op-,YX), $$%
	where $Y\colon \C\op\to[\C,\V]$ is the Yoneda embedding. This is a cocylinder for the diagram $(M,YG\op)$ and therefore corresponds to an arrow $f_\lambda\colon M*YG\op\to YX$ in $[\C,\V]$. It is now easy to see that a $\V$-functor $F\colon \C\to\V$ sends $\lambda$ to a limiting cylinder if and only if it is orthogonal with respect to $f_\lambda$. In conclusion, define $\M$ as the collection of all the $f_\lambda$ for each $\lambda\in\mathbb{L}$; then $\M$ defines an $\alpha$-orthogonality class which coincides with $\tx{Mod}(\S,\V)$.
	
	$(3)\Rightarrow (1)$. Assume that $\V_0$ is locally bounded and let $\M$ be a family of morphisms which defines $\K=\M^{\perp}$ as a $\Phi$-orthogonality class in $[\C,\V]$. It follows at once that $\K$ is closed in $[\C,\V]$ under $\Phi$-flat colimits and all limits. Then by \cite[Theorem~6.5]{Kel82:libro} (since $\V_0$ is assumed to be locally bounded) the inclusion $J\colon\K\hookrightarrow[\C,\V]$ has a left adjoint $F$. It follows that $\K$ is locally $\Phi$-presentable by Theorem~\ref{1stchar}.
\end{proof}

We conclude with some results about the existence of certain limits in the 2-category $\bo{L}\Phi\bo{P}$, generalizing Bird's result~\cite{Bir84:tesi} for locally $\alpha$-presentable $\V$-categories.

\begin{lema}\label{functors}
	If $\K$ is locally $\Phi$-presentable and $\E$ is a small $\V$-category, then $[
	\E,\K]$ is also locally $\Phi$-presentable.
\end{lema}
\begin{proof}
	By Theorem~\ref{1stchar} there exists a small $\V$-category $\C$ and fully faithful $J\colon \K\to[\C,\V]$ which preserves $\Phi$-flat colimits and has a left adjoint $L$. It follows that the $\V$-functor 
	$$[\E,J]\colon[\E,\K]\hookrightarrow[\E,[\C,\V]]\cong[\E\otimes \C,\V]$$
	is fully faithful, preserves $\Phi$-flat colimits (since these are computed pointwise and $J$ preserves them), and has $[\E,L]$ as left adjoint. Thus, by Theorem~\ref{1stchar} the $\V$-category $[\E,\K]$ is locally $\Phi$-presentable.
\end{proof}

In the lemma below we say that a $\V$-functor is an isofibration if its underlying ordinary functor is one.

\begin{lema}\label{limits}
	The forgetful 2-functor $\bo{L}\Phi\bo{P}\to \V\tx{-}\bo{CAT}$ creates (reflects and preserves) the following 2-dimensional limits:\begin{enumerate}
		\item small products;
		\item powers;
		\item if $\V_0$ is locally presentable, pullbacks along isofibrations.
	\end{enumerate}
\end{lema}
\begin{proof}
	(1). Let $K_j$, for $j\in\J$, be a family of locally $\Phi$-presentable $\V$-categories. By Theorem~\ref{1stchar} for each $j$ we find small $\V$-category $\C_j$ together with a fully faithful $H_j\colon K_j\to [\C_j,\V]$ which preserves $\Phi$-flat colimits and has a left adjoint $L_j$.
	Consider now the product $\textstyle\prod_{j\in J}\K_j$ in  $\V\tx{-}\bo{CAT}$, then we have a fully faithful inclusion
	$$  (H_j)_{j\in J}\colon\prod_{j\in J}\K_j\hookrightarrow \prod_{j\in J}[\C_j,\V]\cong[\sum_{j\in J}\C_j,\V]$$
	with left adjoint $(L_j)_{j\in J}$. It is easy to see that limits and colimits are computed componentwise in the products above; thus the component projections of $\textstyle\prod_{j\in J}\K_j$ as well as $(H_j)_{j\in J}$ preserve all limits and all $\Phi$-flat colimits. By Theorem~\ref{1stchar} the product $\textstyle\prod_{j\in J}\K_j$ is locally $\Phi$-presentable; this is enough to show that $\textstyle\prod_{j\in J}\K_j$ is also a product in $\bo{L}\Phi\bo{P}$.
	
	(2). This follows from Lemma~\ref{functors} above since the power of $\K$ by a category $\B$ in $\V\tx{-}\bo{CAT}$ can be obtained simply by taking the $\V$-category $[\B_\V,\K]$ where $\B_\V$ is the free $\V$-category on $\B$.
	
	(3).  Consider $F_1\colon \K_1\to \L$ and $F_2\colon\K_2\to\L$ in $\bo{L}\Phi\bo{P}\to \V\tx{-}\bo{CAT}$ where $F_2$ is an isofibration. Consider their pullback $\M$ together with the projection $G_i\colon\M\to\K_1$.  By \cite[Theorem~5.5]{LT22:virtual} the $\V$-category $\M$ is accessible; moreover, since the pullback is also a bipullback ($F_2$ is an isofibration) and $F_1$ and $F_2$ are continuous and preserve $\Phi$-flat colimits; then $\M$ is complete and $\Phi$-cocomplete and $G_1$ and $G_2$ preserve such limits and colimits. It follows that $\M$ is locally presentable and $G_1\times G_2\colon\M\to \K_1\times\K_2$ is conservative, preserves $\Phi$-flat colimits, and is a right adjoint (being continuous between locally presentable $\V$-categories). Thus $\M$ is locally $\Phi$-presentable by Proposition~\ref{conservative}. This implies that $\M$ is also the pullback of $F_1$ and $F_2$ in $\bo{L}\Phi\bo{P}\to \V\tx{-}\bo{CAT}$.
\end{proof}

\begin{teo}\label{flexible}
	If $\V_0$ is locally presentable, then $\bo{L}\Phi\bo{P}$ has, and $U\colon\bo{L}\Phi\bo{P}\to \V\tx{-}\bo{CAT}$ preserves, all flexible (and hence all pseudo- and bi-) limits.
\end{teo}
\begin{proof}
	Follows directly from Lemma~\ref{limits} above since products, powers, and pullbacks along isofibrations, together with splittings of idempotent equivalences (which exist by \cite[Remark~7.6]{BKPS89:articolo}), imply the existence of all flexible limits (argue as in the beginning of \cite[Theorem~5.5]{LT22:virtual}).
\end{proof}

\subsection{Sound classes induced from $\bo{Set}$}\label{sound-set}$ $

Let $\DD$ be a locally small and weakly sound class of indexing categories in $\bo{Set}$. We say that a small category $\C$ is $\DD$-filtered if $\Delta 1\colon \C\op\to\bo{Set}$ is $\DD$-flat; that is: if $\C$-colimits commute in $\bo{Set}$ with $\DD$-limits. 

In the unenriched setting the definition of weakly sound class further simplifies to the condition below. One direction is already given by \cite[2.5]{ABLR02:articolo} (with the usual caveat that they use sound classes rather than weakly sound ones); the other follows easily.

\begin{prop}\label{simpler}
	The class $\DD$ is weakly sound if and only if every small and $\DD$-cocomplete category $\C$ is $\DD$-filtered.
\end{prop}
\begin{proof}
	Assume that $\DD$ is weakly sound and that $\C$ is $\DD$-cocomplete; then $\Delta 1\colon \C\op\to\bo{Set}$ is $\DD$-continuous and hence $\DD$-flat. Therefore $\C$ is $\DD$-filtered. Conversely, assume that every small and $\DD$-cocomplete category $\C$ is $\DD$-filtered, and let $M\colon \C\op\to\bo{Set}$ be $\DD$-cocontinuous. It follows that $\tx{El}(M)\op$ is $\DD$-cocomplete and hence $\DD$-filtered by hypothesis. As a consequence $M$ is a $\DD$-filtered colimit of representables, and thus $\DD$-flat.
\end{proof}

Recall also the following.

\begin{prop}\cite[Theorem~2.4]{ABLR02:articolo}
	Let $\DD$ be a weakly sound class. A weight $M\colon \C\op\to\bo{Set}$ is $\DD$-flat if and only if the category $\tx{El}(M)\op$ is $\DD$-filtered.
\end{prop}

Let $\V=(\V_0,\otimes,I)$ by a locally $\DD$-presentable and symmetric monoidal closed category for which $(\V_0)_\DD$ is closed under tensor product (but does not necessarily contain the unit); in particular $\V_0$ is locally presentable.

\begin{Def}(\cite[Section~5]{LR11NotionsOL})
	Denote by $\Phi_\DD$ the locally small class of weights in $\V$ generated by the conical weights in $\DD$ and powers by objects in $(\V_0)_\DD$.
\end{Def}

\begin{es}\label{inducedfromSet}
	When $\V$ is locally $\alpha$-presentable and $\DD$ consists of the $\alpha$-small diagrams, then $\Phi_\DD $-limits generate the $\alpha$-small weighted limits of Kelly \cite{Kel82:articolo}. If $\V$ is a quasivariety and $\DD$ is the class for finite products, $\Phi_\DD $-limits and colimits have been considered in Section~5 of \cite{LR11NotionsOL} where $\V$ is called locally strongly finitely presentable as a $\otimes$-category; this is also called symmetric monoidal quasivariety in~\cite{LT20:articolo}.
\end{es}

\begin{prop}\tx{(\cite[5.22]{LR11NotionsOL}).}\label{Phi-V-sound}
	Let $\C$ be small and $\Phi_\DD$-cocomplete. The following are equivalent for a weight $M\colon \C\op\to\V$:\begin{enumerate}\setlength\itemsep{0.25em}
		\item $M$ is $\Phi_\DD $-flat;
		\item $M$ is $\Phi_\DD $-continuous;
		\item $M$ is a conical $\DD$-filtered colimit of representables.
	\end{enumerate}
	Moreover, $\DD$-filtered conical colimits commute in $\V$ with $\Phi_\DD $-limits. It follows in particular that the class $\Phi_\DD $ is weakly sound.
\end{prop}

Since $\Phi_\DD$ is weakly sound it follows that all the results of Section~\ref{main-presentable} apply, including the duality of Theorem~\ref{bieq}. Next, we will try to understand the enriched notion of $\Phi_\DD$-presentability by just using (conical) $\DD$-filtered colimits. 

\begin{Def}
	We say that an object $X$ of a $\V$-category $\K$ is {\em $\DD$-presentable} if the $\V$-functor $\K(X,-)\colon \K\to\V$ preserves (conical) $\DD$-filtered colimits. Denote by $\K_\DD$ the full subcategory of $\K$ spanned by the $\DD$-presentable objects. We say that a cocomplete $\V$-category $\K$ is {\em locally $\DD$-presentable} if it has a strong generator made of $\DD$-presentable objects.
\end{Def}

This is a priori different from local $\Phi_\DD$-presentability, but in the cocomplete case this difference vanishes, as the following proposition demonstrates.

\begin{prop}\label{flat-cocomplete}
	Let $\K$ be a complete or $\Phi_\DD$-cocomplete $\V$-category; then $\K$ has $\Phi_\DD$-flat colimits if and only if it has $\DD$-filtered colimits. A $\V$-functor from such a $\K$ preserves $\Phi_\DD$-flat colimits if and only if it preserves $\DD$-filtered colimits.
\end{prop}
\begin{proof}
	The proof is exactly the same as that of~\cite[Proposition~3.20]{LT22:virtual}, simply replace $\alpha$-small colimits with $\Phi_\DD$-colimits and use Proposition~\ref{Phi-V-sound} above.
\end{proof}

\begin{cor}\label{DvsPhiD}
	A $\V$-category $\K$ is locally $\Phi_\DD$-presentable if and only if it is locally $\DD$-presentable.
\end{cor}
\begin{proof}
	Since $\K$ is cocomplete, this follows at once by the previous proposition because then $\DD$-presentable and $\Phi_\DD$-presentable objects coincide.
\end{proof}

To compare the enriched with the ordinary notion of presentability we need a further assumption on the ordinary $\DD$-presentable objects of $\V_0$:

\begin{Def}
	We say that $\V$ is {\em locally $\DD$-presentable as a closed category} if $\V_0$ is locally $\DD$-presentable and $(\V_0)_\DD$ contains the unit and is closed under tensor product.
\end{Def}

\begin{prop}
	Let $\V$ be locally $\DD$-presentable as a closed category. If $\K$ is locally $\Phi_\DD$-presentable then $\K_0$ is locally $\DD$-presentable as an ordinary category. Moreover 
	$$ (\K_{\Phi_\DD})_0=(\K_{\DD})_0=(\K_0)_\DD. $$
\end{prop}
\begin{proof}
	We already know from the previous proposition that $(\K_{\Phi_\DD})_0=(\K_{\DD})_0$. Since the unit is $\DD$-presentable then every $\DD$-presentable object in $\K$ is also $\DD$-presentable in $\K_0$; thus $(\K_{\DD})_0\subseteq(\K_0)_\DD$. This is enough to guarantee that $(\K_{\DD})_0$ generates $\K_0$ under $\DD$-filtered colimits and is made by $\DD$-presentable objects; hence $\K_0$ is locally $\DD$ presentable. It follows now that each element of $(\K_0)_\DD$ is a $\DD$-colimit of objects from $(\K_{\DD})_0$, which is closed under these colimits. Thus $(\K_{\DD})_0=(\K_0)_\DD$.
\end{proof}

Denote by $\Q$ the class of Cauchy weights and by $\S$ the one-element class for split coequalizers. 

\begin{prop}\label{Cauchy-sound}
	Let $\V$ be locally $\DD$-presentable as a closed category. Then $$\Q\subseteq (\Phi_\DD\cup\S)^*.$$
\end{prop}
\begin{proof}
	It is enough to prove that every small $\Phi_\DD$-complete $\V$-category with splittings of idempotents is Cauchy complete as a $\V$-category; the fact that every $\Phi_\DD$-continuous $\V$-functor also preserves Cauchy limits is trivial. Let $\C$ be small and $\Phi_\DD$-cocomplete with splittings of idempotents, and let $\K:=\Phi_\DD\tx{-Cont}(\C\op,\V)$; then $\K$ is locally $\Phi_\DD$-presentable and $\C\subseteq \K_{\Phi_\DD}$ generates it under $\DD$-filtered colimits. By the previous Proposition the ordinary category $\K_0$ is locally $\DD$-presentable too with $(\K_{\Phi_\DD})_0=(\K_0)_\DD$. Note now that $\C_0$ has split idempotents, is a strong generator of $\K_0$ (since it generates under conical colimits), and is closed in $\K_0$ under $\DD$-colimits; thus $\C_0\simeq(\K_0)_\DD$ by the ordinary Gabriel-Ulmer duality for $\DD$. It follows that $\C\simeq \K_{\Phi_\DD}$ and is therefore Cauchy complete.
\end{proof}

A direct consequence is the following: 

\begin{cor}
	Let $\V_0$ be locally $\alpha$-presentable as a closed category. Then the Cauchy weights are (in the saturation of the) $\alpha$-small weights.
\end{cor}

Since the class of $\alpha$-small weights is pre-saturated~\cite[Definition~A.1.13]{Ten22:phd}, this means that every Cauchy weight $M$ is the left Kan extension of some $\alpha$-small weight~\cite[Proposition~A.1.14]{Ten22:phd}.

As a consequence we recover the following result of Johnson:

\begin{cor}[\cite{Joh1989:articolo}]\label{Cauchy-small}
	Let $\V_0$ be locally presentable. Then the class $\Q$ is essentially small.
\end{cor}
\begin{proof}
	If $\V_0$ is locally presentable then it is also locally $\alpha$-presentable as a closed category for some $\alpha$ by \cite[Proposition~2.4]{KL2001:articolo}. Thus the result follows from the corollary above since the class of $\alpha$-small weights is essentially small.
\end{proof}

Finally we can prove:

\begin{cor}
	Let $\V$ be locally $\DD$-presentable as a closed category and such that $\Phi_\DD$-colimits are absolute (equivalently, conical $\DD$-colimits and $\V_\DD$-powers are absolute). Then 
	$$\Q=(\Phi_\DD\cup\S)^*.$$
\end{cor}
\begin{proof}
	This is a direct consequence of Proposition~\ref{Cauchy-sound} since, by hypothesis, also the other inclusion holds.
\end{proof}

\begin{obs}
	The corollary above applies when $\V$ is locally dualizable category as in~\cite[Definition~4.3]{LT21:articolo}. Indeed, one needs to take $\DD$ to be the class of finite products; then $\Phi_\DD$-colimits are generated by finite direct sums and powers by a strong generator made of dualizable objects. Thus, as a consequence of the Corollary above we recover~\cite[Corollary~4.22]{LT21:articolo}, and in particular the characterization of absolutely complete $\Ab$-categories and $\bo{GAb}$-categories (\cite{NST2020cauchy}).
\end{obs}

\section{Soundness and raising the index of accessibility}\label{section-raising}

Unlike the locally presentable case, raising the index of accessibility is an issue: ordinarily it is not true that if $\A$ is an $\alpha$-accessible category and $\beta>\alpha$, then $\A$ is $\beta$-accessible. This problem was address from the very beginning~\cite{MP89:libro,AR94:libro} by introducing a new order relationship $\triangleleft$, called the {\em sharply less relationship}, between regular cardinals. In \cite[Section~2.3]{MP89:libro} it is proved that, given $\alpha<\beta$, every $\alpha$-accessible category $\A$ is $\beta$-accessible if and only if $\alpha\triangleleft\beta$, and in that case every $\beta$-presentable object of $\A$ is a $\beta$-small $\alpha$-filtered colimit of $\alpha$-presentable objects.

In the context of a sound class of weights this problem has never been addressed (not even for ordinary categories); all that is known is that every $\Phi$-accessible $\V$-category $\A$ is $\alpha$-accessible for arbitrarily large regular cardinals $\alpha$ \cite[Theorem~4.4]{LT22:limits}, but such $\alpha$ might depend on the given $\V$-category $\A$. In this section we shall extend the sharply less relationship to the context of sound classes of weights and prove theorems similar to those known in the ordinary setting.

\begin{as}
	For the purposes of this section all the classes of weights are assumed to be saturated; thus there is no distinction between weakly sound and sound classes (Remark~\ref{sat-sound}). We also assume $\V_0$ to be locally presentable: while Theorem~\ref{sharply} holds without the local presentability assumption, this is needed to prove Theorem~\ref{rainsing}.
	
\end{as}

Given a regular cardinal $\alpha$ for which $\V_0$ is locally presentable as a closed category, we denote by $\P_\alpha$ the sound class given by the saturation of the $\alpha$-small weights.

\begin{Def}
	Given two sound classes $\Phi\subseteq\Psi$, we define:
	$$ \Phi^{^+}_\Psi:=\Phi\+ \cap \Psi .$$
\end{Def}

\begin{obs}
	Since $\Phi\+$ and $\Psi$ are saturated, for any small $\C$ we can see $\Phi^{^+}_\Psi\C$ as the intersection below.
	\begin{center}
		\begin{tikzpicture}[baseline=(current  bounding  box.south), scale=2]
			
			\node (a0) at (0,0.8) {$\Phi^{^+}_\Psi\C$};
			\node (a0') at (0.3,0.6) {$\lrcorner$};
			\node (b0) at (1.2,0.8) {$\Phi\+\C$};
			\node (c0) at (0,0) {$\Psi\C$};
			\node (d0) at (1.2,0) {$[\C\op,\V]$};
			
			\path[font=\scriptsize]
			
			(a0) edge [right hook->] node [above] {} (b0)
			(a0) edge [right hook->] node [left] {} (c0)
			(b0) edge [right hook->] node [right] {} (d0)
			(c0) edge [right hook->] node [below] {} (d0);
		\end{tikzpicture}	
	\end{center} 
	Indeed, by saturatedness, a weight $M\colon\C\op\to\V$ is in $\Phi\+$ (respectively, in $\Psi$) if and only if it lies in $\Phi\+\C$ (respectively, in $\Psi\C$).
\end{obs}

\begin{obs}
	When $\Phi=\Psi$, then 
	$$ \Phi^{^+}_\Phi\C\simeq\Q\C$$
	coincides with the Cauchy completion of $\C$. This follows from a straightforward generalization of \cite[Lemma~3.3]{Ten2022continuity} where instead of $\alpha$-small weights one uses the sound class~$\Phi$.
\end{obs}

Now, for any small $\V$-category $\C$ we have an inclusion
\begin{equation}\label{sharp-inclusion}
	 \Psi\+(\Phi^{^+}_\Psi\C) \subseteq \Phi\+\C
\end{equation}
as full subcategories of $\P\C$. Indeed, $\Psi\+(\Phi^{^+}_\Psi\C)$ is contained in $\P\C$ by \cite[4.10]{LT22:limits} since $\Psi\+$-colimits commute in $\V$ with $\Phi^{^+}_\Psi$-limits (these limits are in particular $\Psi$-limits), and then $\Psi\+(\Phi^{^+}_\Psi\C) \subseteq \Phi\+\C$ since both $\Psi\+$ and $\Phi^{^+}_\Psi$ are contained in $\Phi\+$. 

\begin{prop}
	Consider sound classes $\Phi\subseteq\Psi$, then any locally $\Phi$-presentable $\K$ is locally $\Psi$-presentable. 
\end{prop}
\begin{proof}
	Since every $\Phi$-presentable object is $\Psi$-presentable, a strong generator made of $\Phi$-presentable objects is also a strong generator made of $\Psi$-presentable objects. Thus local $\Phi$-presentability implies local $\Psi$-presentability.
\end{proof}

The same does not hold for $\Phi$-accessible and $\Psi$-accessible $\V$-categories, even in the ordinary case. That is why the {sharply less than} relation between regular cardinal was introduced in~\cite{MP89:libro}.

Given weakly sound classes $\Phi\subseteq\Psi$ and a $\Phi$-accessible $\V$-category $\A$, we can consider the full subcategories $\A_\Phi\subseteq\A_\Psi$ of $\A$ spanned respectively by the $\Phi$-presentable and $\Psi$-presentable objects of $\A$. By $\Phi$-accessibility of $\A$ and definition of $\Phi^{^+}_\Psi$ we know that
$$ \A_\Phi\subseteq \Phi^{^+}_\Psi\A_\Phi\subseteq \Phi\+\A_\Phi\simeq \A. $$
Then, since the $\Psi$-presentable objects are closed under all existing $\Psi$-colimits, it follows that we have an inclusion
\begin{equation}\label{Phi-Psi-pres}
	 \Phi^{^+}_\Psi\A_\Phi\subseteq \A_\Psi
\end{equation}
as full subcategories of $\A$. In particular, when $\A=[\C\op,\V]$ is the presheaf $\V$-category on a small $\C$, then $\A_\Phi=\Phi\C$ (since it is the closure of the representable under $\Phi$-colimits) and $\A_\Psi=\Psi\C$; therefore the inclusion above can be rewritten as
\begin{equation}\label{Phi-Psi-free}
	 \Phi^{^+}_\Psi(\Phi\C)\subseteq \Psi\C.
\end{equation}
We are now ready to state and prove the following result which generalizes \cite[Theorem~2.3]{LT22:limits} to general sound classes of weights.

\begin{teo}\label{sharply}
	The following are equivalent for any given sound classes $\Phi\subseteq\Psi$: \begin{enumerate}\setlength\itemsep{0.25em}
		\item For any small $\V$-category $\C$ the inclusion~(\ref{sharp-inclusion}) above is an equivalence $$ \Psi\+(\Phi^{^+}_\Psi\C) \simeq \Phi\+\C.$$
		\item If $M\colon\C\op\to\V$ is a $\Phi$-flat weight, then $\tx{Ran}_{J\op}M\colon(\Phi^{^+}_\Psi\C)\op\to\V$ is $\Psi$-flat, where $J\colon\C\to\Phi^{^+}_\Psi\C$ is the inclusion.
		\item Every $\Phi$-accessible $\V$-category $\A$ is $\Psi$-accessible and the inclusion~(\ref{Phi-Psi-pres}) above is an equivalence  $$ \Phi^{^+}_\Psi\A_\Phi\simeq \A_\Psi.$$
		\item Every $\Phi$-accessible $\V$-category $\A$ is $\Psi$-accessible, and for any $\Phi\+$-cocontinuous $\V$-functor $F\colon\A\to \B$, between $\Phi$-accessible $\V$-categories, if $F(\A_\Phi)\subseteq\B_\Phi$ then also $F(\A_\Psi)\subseteq\B_\Psi$.
		\item For any small $\C$ the $\V$-category $\Phi\+\C$ is $\Psi$-accessible and the inclusion $J\colon \Phi\+\C\hookrightarrow[\C\op,\V]$ preserves the $\Psi$-presentable objects.
	\end{enumerate}
	When they hold, we say that {\em $\Phi$ is sharply less than $\Psi$}, and write $\Phi\trianglelefteq \Psi$.
\end{teo}
\begin{proof}
	$(1)\Rightarrow(2)$. Let $M\colon\C\op\to\V$ be $\Phi$-flat, so that $M\in\Phi\+\C\simeq \Psi\+(\Phi^{^+}_\Psi\C) = \Phi\+\C$. As a consequence $M$ is a $\Psi$-flat colimit of elements in $\Phi^{^+}_\Psi\C$; that is, there exists $H\colon \D\to \Phi^{^+}_\Psi\C$ and $N\colon \D\op\to \V$ in $\Psi\+$ such that $M\cong N*ZH$, where $Z\colon \Phi^{^+}_\Psi\C\hookrightarrow[\C\op,\V]$ is the inclusion. By replacing $N$ with $\tx{Lan}_{J\op}N$ (since this is still in $\Psi\+$) we can assume that $\D=\Phi^{^+}_\Psi\C$ and $H=1_{\Phi^{^+}_\Psi\C}$, so that $M\cong N*Z$. From this it follows that
	\begin{align*}
		M(-) & \cong \tx{ev}_{(-)}(N*Z)\\ 
		& \cong N\square*\tx{ev}_{(-)}Z\square\\ 
		& \cong N\square*\Phi^{^+}_\Psi\C(J-,\square)\\ 
		& \cong NJ\op(-).
	\end{align*}
	
	Now, since $N\colon (\Phi^{^+}_\Psi\C)\op\to \V$ is $\Psi$-flat, it is in particular $\Phi^{^+}_\Psi$-continuous, and hence $N\cong \tx{Ran}_{J\op}(NJ\op)$. As a consequence
	$$ \tx{Ran}_{J\op}M \cong \tx{Ran}_{J\op}NJ\op\cong N $$
	is $\Psi$-flat.
	
	$(2)\Rightarrow(3)$. Let $\A$ be $\Phi$-accessible. Then $\Phi^{^+}_\Psi\A_\Phi\subseteq \A_\Psi$ by (\ref{Phi-Psi-pres}). Given $A\in\A$ we can write it as a $\Phi$-flat colimit of the form $A\cong \A(Z-,A)*Z$ where $Z\colon \A_\Phi\to \A$ is the inclusion. Let now $K\colon \Phi^{^+}_\Psi\A_\Phi\to \A$ and $J\colon\A_\Phi\to\Phi^{^+}_\Psi\A_\Phi$ also be the inclusions, so that we have $KJ=Z$. Since $\Phi^{^+}_\Psi\A_\Phi$ is dense in $\A$ we also have 
	$$ A\cong \A(K-,A)*K, $$
	but $\A(K-,A)\cong \tx{Ran}_{J\op}\A(Z-,A)$ is $\Psi$-flat by (2). This is enough to conclude that $\A$ is $\Psi$-accessible.
	
	Now, by the proof of \cite[3.9]{LT22:virtual} it follows that $\A_\Psi$ is the Cauchy completion of $\Phi^{^+}_\Psi\A_\Phi$, which is already Cauchy complete. Thus $\A_\Psi\simeq \Phi^{^+}_\Psi\A_\Phi$.
	
	$(3)\Rightarrow(1)$. Given a small $\V$-category $\C$, the free cocompletion $\Phi\+\C$ is $\Phi$-accessible by definition. Thus, by (3), $\Phi\+\C$ is also $\Psi$-accessible and $\Phi\+\C\simeq \Psi\+((\Phi\+\C)_\Psi)$. Now, again by (3), we have
	$$ (\Phi\+\C)_\Psi\simeq \Phi^{^+}_\Psi(\Phi\+\C)_\Phi \simeq \Phi^{^+}_\Psi(\Q\C)) \simeq \Phi^{^+}_\Psi\C $$
	where we used that $(\Phi\+\C)_\Phi\simeq \Q\C$ and that $\Phi^{^+}_\Psi$ contains the Cauchy weights. Putting the equivalences together we obtain that $ \Psi\+(\Phi^{^+}_\Psi\C) \simeq \Phi\+\C$.
	
	$(3)\Rightarrow(4)$. We only need to prove the second part of the statement. Assume that $F\colon\A\to \B$ is a $\Phi\+$-cocontinuous $\V$-functor between $\Phi$-accessible $\V$-categories such that $F(\A_\Phi)\subseteq\B_\Phi$. Then, given any $X\in\A_\Psi$, by (3) we can write $X$ as a $\Phi^{^+}_\Psi$-colimit of elements from $\A_\Phi$; since $F$ preserves $\Phi\+$-colimits it follows that $FX$ is a $\Phi^{^+}_\Psi$-colimit of objects from $\B_\Phi$. Then, since $\B_\Psi$ is closed under any existing $\Psi$-colimits and $\B_\Phi\subseteq\B_\Psi$, it follows that $FX\in\B_\Psi$.
	
	$(4)\Rightarrow(5)$. Assuming (4), the first part of (5) is trivial. For the statement about the inclusion $J\colon \Phi\+\C\hookrightarrow[\C\op,\V]$, notice that this preserves $\Phi\+$-colimits by definition and sends $(\Phi\+\C)_\Phi$ to $\Q\C\subseteq [\C\op,\V]$ which is contained in $[\C\op,\V]_\Phi=\Phi\C$. Thus $J$ preserves the $\Psi$-presentable objects by (4).
	
	$(5)\Rightarrow(3)$. Any $\Phi$ accessible $\A$ can be written as $\A\simeq\Phi\+\A_\Phi$, and thus is $\Psi$-accessible by the first part of (5). Under such equivalence we have $\A_\Psi\simeq (\Phi\+\A_\Phi)_\Psi $. By the second part of (5) it follows that
	$$(\Phi\+\A_\Phi)_\Psi\subseteq [\A_\Phi\op,\V]_\Psi=\Psi(\A_\Phi).$$
	Therefore
	$$ \A_\Psi \subseteq \Phi\+(\A_\Phi)\cap \Psi(\A_\Phi)=\Phi^{^+}_\Psi(\A_\Phi);$$
	since the other inclusion always holds, it follows that $\A_\Psi\simeq \Phi^{^+}_\Psi\A_\Phi$.
\end{proof}

As a consequence:

\begin{cor}
	The sharply less relationship $\triangleleft$ between sound classes of weights is transitive.
\end{cor}
\begin{proof}
	It follows from point (4) of Theorem~\ref{sharply} above.
\end{proof}

\begin{obs}
	When $\Phi$ and $\Psi$ are respectively $\P_\alpha$ and $\P_\beta$, for $\alpha<\beta$, then it is enough to require that $\alpha$-accessibility implies $\beta$-accessibility to obtain that $\P_\alpha\triangleleft\P_\beta$: the additional conditions in (3) and (4) are then automatically satisfied (see~\cite[Proposition~2.3.11]{MP89:libro} for the ordinary result; the enriched version follows from~\cite[Theorem~3.25]{LT22:virtual}). It follows that $\P_\alpha\triangleleft\P_\beta$ if and only if $\alpha\triangleleft\beta$ as regular cardinals. 
\end{obs}

\begin{cor}
	Given sound classes $\Phi\trianglelefteq \Psi$, the inclusion~(\ref{Phi-Psi-free}) is an equivalence
	$$ \Phi^{^+}_\Psi(\Phi\C)\simeq \Psi\C $$
	for any small $\V$-category $\C$.
\end{cor}
\begin{proof}
	Given $\C$, consider the $\Phi$-accessible $\V$-category $\K:=[\C\op,\V]$. Then $\K_\Phi=\Phi\C$ and $\K_\Psi=\Psi\C$. Thus the result follows from point (3) of the theorem above.
\end{proof}

Below, given a regular cardinal $\beta$, we denote $\Phi^{^+}_\beta:=\Phi^{^+}_{\P_\beta}$. 

\begin{teo}\label{rainsing}
	For any weakly sound class $\Phi$ there exists a regular cardinal $\alpha$ for which $$\Phi\triangleleft\P_\alpha.$$ Such an $\alpha$ can be chosen to be any regular cardinal sharply greater than some given $\beta$ which satisfies $\Phi^{^+}_\beta(\Phi\C)\simeq \P_\beta\C$, induced by the inclusion~(\ref{Phi-Psi-pres}), for any small $\V$-category $\C$.
\end{teo}
\begin{proof}
	Given $\Phi$, the first step of the proof is to find a $\beta$ for which $\Phi^{^+}_\beta(\Phi\C)\simeq \P_\beta\C$ for any small $\V$-category $\C$.
	
	For that, consider the first $\beta'$ such that $\Phi\subseteq\P_{\beta'}$. Then, there exists $\beta\triangleright\beta'$ for which $\P_{\beta'}\C\subseteq \Phi^{^+}_\beta(\Phi\C)$ for any $\C$. Indeed, fix a small set $S$ of weights whose saturation is $\P_{\beta'}$, then  since $\Phi\+(\Phi(-))\simeq\P(-)$ (by \cite[Proposition~4.9]{LT22:limits}), it follows that any weight $M\in S$ can be written as a $\Phi$-flat colimit of $\Phi$-colimits of representables. The $\Phi$-flat colimit involved in such description of $M$ will be $\beta_M$-small for sone $\beta_M$; thus it is enough to consider $\beta\geq\tx{sup}_{M\in S}\beta_M$. It is now clear that for such $\beta$ we have $\P_{\beta'}\C\subseteq \Phi^{^+}_\beta(\Phi\C)$ for any $\C$.
	
	Now, since $\Phi\subseteq\P_{\beta'}$ and $\beta'<\beta$ it follows that the inclusion
	$$ \Phi^{^+}_\beta(\Phi\C)\subseteq \P_\beta\C $$
	holds for any $\C$. For the converse note that, since $\beta\triangleright\beta'$, all $\beta$-small conical colimits are generated by the $\beta'$-small and the $\beta$-small $\beta'$-filtered conical ones. It follows that $\beta$-small copowers can also be written as $\beta$-small $\beta'$-filtered conical colimits of $\beta'$-small copowers. Therefore, any object of $\P_\beta\C$ (that is, any $\beta$-small weighted colimit of representables) can be rewritten as a $\beta$-small $\beta'$-filtered colimit from $\P_{\beta'}\C$. But $\P_{\beta'}\C\subseteq \Phi^{^+}_\beta(\Phi\C)$ and $\beta$-small $\beta'$-filtered colimits are in $\Phi^{^+}_\beta$ (since $\Phi\subseteq\P_{\beta'}$). Thus $ \Phi^{^+}_\beta(\Phi\C)=\P_\beta\C $.
	
	For the main objective of this theorem, consider then any $\alpha\triangleright\beta$. We shall prove that condition~(5) of Theorem~\ref{sharply}  above holds, so that $\Phi\triangleleft\P_\alpha$.
	
	The argument is similar to that of \cite[3.25]{LT22:virtual}. Let $\C$ be any small $\V$-category, then we can see $\Phi\+\C$ as the intersection
	\begin{center}
		\begin{tikzpicture}[baseline=(current  bounding  box.south), scale=2]
			
			\node (a0) at (0,0.9) {$\Phi\+\C$};
			\node (b0) at (1.5,0.9) {$[\C\op,\V]$};
			\node (c0) at (0,0) {$[\C\op,\V]$};
			\node (d0) at (1.5,0) {$[(\Phi\C)\op,\V]$};
			\node (e0) at (0.2,0.7) {$\lrcorner$};
			
			\path[font=\scriptsize]
			
			(a0) edge [right hook->] node [above] {$J$} (b0)
			(a0) edge [right hook->] node [left] {$J$} (c0)
			(b0) edge [right hook->] node [right] {$\tx{Lan}_{H\op}$} (d0)
			(c0) edge [right hook->] node [below] {$\tx{Ran}_{H\op}$} (d0);
		\end{tikzpicture}	
	\end{center}
	where $H\colon\C\hookrightarrow\Phi\C$ is the inclusion. Indeed, a $\V$-functor $M\colon\C\op\to\V$ lies in $\Phi\+\C$ if and only if $M$ is $\Phi$-flat, if and only if $\tx{Lan}_{H\op}M$ is $\Phi$-flat (by \cite[3.3]{LT22:virtual}), if and only if $\tx{Lan}_{H\op}M$ is $\Phi$-continuous (by soundness of $\Phi$), if and only if $\tx{Lan}_{H\op}M\cong \tx{Ran}_{H\op}M$ (since $(\Phi\C)\op$ is the free completion of $\C\op$ under $\Phi$-colimits).
	
	By~\cite[Lemma~3.24]{LT22:virtual} and its proof, to show that $\Phi\+\C$ is $\alpha$-accessible and that $J$ preserves the $\alpha$-presentable objects, it is enough to show that $\tx{Ran}_{H\op}$ and $\tx{Lan}_{H\op}$ both preserve $\beta$-flat colimits and the $\beta$-presentable objects.
	
	Note that the $\beta$-presentable objects of $[\C\op,\V]$ are spanned by $\P_\beta\C$. On the one hand, $\tx{Lan}_{H\op}$ is cocontinuous and sends representables to representables; therefore it sends $\P_\beta\C$ to $\beta$-small colimits of representables in $[(\Phi\C)\op,\V]$, which are $\beta$-presentable.\\
	On the other hand, since $\tx{Ran}_{H\op}$ is defined by computing certain $\Phi$-limits in $\V$, it therefore preserves all $\Phi$-flat colimits. Concerning preservation of the $\beta$-presentable objects, we shall use that $\P_\beta\C = \Phi^{^+}_\beta(\Phi\C)$. Now, $\tx{Ran}_{H\op}$ sends $\Phi\C$ to the representables in $[(\Phi\C)\op,\V]$ (by construction) and thus, since it preserves $\Phi^{^+}_\beta$-colimits (these being $\Phi$-flat), it sends $\P_\beta\C = \Phi^{^+}_\beta(\Phi\C)$ to $\Phi^{^+}_\beta$-colimits of representables in $[(\Phi\C)\op,\V]$, which are $\beta$-presentable objects (since $\Phi^{^+}_\beta\subseteq\P_\beta$).
\end{proof}

We now turn to some examples:

\begin{ese}\label{exe-sharp}$ $
	\begin{enumerate}
		\item If $\Phi=\P_\lambda$ we can choose $\beta=\lambda$; hence we obtain that $\P_\lambda\triangleleft\P_\alpha$ whenever $\lambda\triangleleft\alpha$. (This is also a consequence of \cite[3.25]{LT22:virtual}.)
		
		\item If $\V=\Set$ and $\Phi=\textbf{Fp}$ is the class for finite products, then $\textbf{Fp}^+_\omega(\textbf{Fp}\ \C)=\textbf{Fin}\ \C$ for any $\C$: one can obtain all finite colimits by adding first finite coproducts and then coequalizers of reflexive pairs. Therefore $\textbf{Fp}\triangleleft \P_{\aleph_1}$.
		(This was also observed in \cite[4.8]{AR11:articolo}.) It remains an open problem, even for $\V=\Set$, whether $\textbf{Fp}\triangleleft \P_{\aleph_0}$.
		
		\item If $\V=\bo{Cat}$ and $\Phi=\textbf{Fp}$ is again the class for finite products, then $\textbf{Fp}^+_\omega(\textbf{Fp}\ \C)=\textbf{Fin}\ \C$ for any $\C$ by Corollary~\ref{fp<fin}. Therefore $\textbf{Fp}\triangleleft \P_{\aleph_1}$.
		
		\item If $\V$ is a symmetric monoidal variety as in \cite{LT20:articolo} and $\Phi=\textbf{Fpp}$ is the class for finite products and powers by finitely presentable projective objects, then $\textbf{Fpp}^+_\omega(\textbf{Fpp}\ \C)=\textbf{Fin}\ \C$ for any $\C$: indeed, conical sifted colimits commute in $\V$ with $\textbf{Fpp}$-limits, and hence reflexive coequalizers are in $\textbf{Fpp}^+_\omega$. Thus one can argue again by saying that finite weighted colimits can be obtained from $\textbf{Fpp}$-limits by adding reflexive coequalizers. Therefore $\textbf{Fpp}\triangleleft \P_{\aleph_1}$.
		
		\item If $\V=\Set$ and $\Phi=\lambda\tx{-}\textbf{Conn}$ is the class for $\lambda$-small connected limits, then $\lambda\tx{-}\textbf{Conn}^+_\lambda(\lambda\tx{-}\textbf{Conn}\ \C)=\P_\lambda \C$ for any $\C$: note that $\lambda\tx{-}\textbf{Conn}^+$ contains all $\lambda$-small coproducts; thus it is enough to observe that any colimit of a $\lambda$-small diagram can be written as a $\lambda$-small coproduct of the ($\lambda$-small) colimits of its connected components. Therefore $\lambda\tx{-}\textbf{Conn}\triangleleft \P_{\lambda^+}$.
		
	\end{enumerate}
\end{ese}

\begin{cor}\label{arbitr-large}
	For any sound class $\Phi$ there are arbitrarily large regular cardinals $\lambda$ for which $\Phi\triangleleft\P_\lambda$.
\end{cor}
\begin{proof}
	This follows from the theorem above, the fact that $\triangleleft$ is transitive, and that, given $\alpha$, there are arbitrarily large $\lambda$ with $\alpha\triangleleft\lambda$ (see \cite[Corollary~2.3.6]{MP89:libro}).
\end{proof}

\section{Soundness and universal algebra}\label{UAsection}

In this section we extend the work of~\cite{RT23EUA} to the setting of weakly sound classes of weights. Indeed, while in~\cite{RT23EUA} the aim was to give a notion of $\lambda$-ary language and $\lambda$-ary equational theory (based on recursively generated terms) whose models characterize the $\V$-category of algebras of $\lambda$-ary monads on $\V$; here we extend these notions to obtain a characterization of the $\V$-category of algebras of those monads on $\V$ which preserve $\Phi$-flat colimits, for a given weakly sound class $\Phi$ of weights.

We start by recalling the notion of languages, structures, and terms introduced in \cite{RT23EUA}. The same notions of languages and structures were considered before in~\cite[Section~5]{LP}.

\begin{Def}[\cite{LP}]
	A single-sorted {\em (functional) language} $\mathbb L$ (over $\V$) is the data of a set of function symbols $f\colon(X,Y)$ whose arities $X$ and $Y$ are objects of $\V$.
\end{Def}

\begin{Def}[\cite{LP}]
	Given a language $\mathbb L$, an {\em $\mathbb L$-structure} is the data of an object $A\in\V$ together with a morphism $$f_A\colon A^X\to A^Y$$ in $\V$ for any function symbol $f\colon(X,Y)$ in $\mathbb L$.
	
	A {\em morphism of $\mathbb L$-structures} $h\colon A\to B$ is the data of a map $h\colon A\to B$ in $\V$ making the following square commute
	\begin{center}
		\begin{tikzpicture}[baseline=(current  bounding  box.south), scale=2]
			
			\node (a0) at (0,0.8) {$A^X$};
			\node (b0) at (1,0.8) {$B^X$};
			\node (c0) at (0,0) {$A^Y$};
			\node (d0) at (1,0) {$B^Y$};
			
			\path[font=\scriptsize]
			
			(a0) edge [->] node [above] {$h^X$} (b0)
			(a0) edge [->] node [left] {$f_A$} (c0)
			(b0) edge [->] node [right] {$f_B$} (d0)
			(c0) edge [->] node [below] {$h^Y$} (d0);
		\end{tikzpicture}	
	\end{center} 
	for any $f\colon(X,Y)$ in $\mathbb L$.
\end{Def}

The $\V$-category $\Str(\LL)$ of $\LL$-structured is defined in \cite[Definition~3.3]{RT23EUA}; it has $\LL$-structures as objects and morphisms of such as arrows. This comes together with a forgetful $\V$-functor $U\colon \Str(\LL)\to \V$ which has a left adjoint $F$.

The recursively generated terms generalize those introduced in \cite[Definition~4.1]{RT23EUA}, while extended terms were already introduced in \cite[Definition~3.2]{LP} (see \cite[Remark~4.6]{RT23EUA}).

\begin{Def}\label{0-terms}
	Let $\mathbb L$ be a language over $\V$. The class of \textit{$\mathbb L$-terms} is defined recursively as follows:
	\begin{enumerate}
		\item Every morphism $f\colon Y\to X$ of $\V$ is an $(X,Y)$-ary term; 
		\item Every function symbol $f:(X,Y)$ of $\mathbb L$ is an $(X,Y)$-ary term;
		\item If $t$ is a $(X,Y)$-ary term and $Z$ is in $\V$, then $t^Z$ is a $(Z\otimes X,Z\otimes Y)$-ary term;
		\item Given $t_J=(t_j)_{j\in J}$, where $J$ is small and $t_j$ is an $(X_j,Y_j)$-ary term, and $s$ an $(\sum_{j\in J} Y_j,W)$-ary term; then $s(t_J)$ is a $(\sum_{j\in J} X_j, W)$-ary term.
	\end{enumerate}
	An {\em extended term $t:(X,Y)$ for $\mathbb L$} is a morphism $t\colon FY\to FX$ in $\Str(\mathbb L)$.
\end{Def}

Given a (possibly extended) term $t:(X,Y)$ and $A\in\Str(\LL)$, we define the interpretation $t_A\colon A^X\to A^Y$ of $t$ in $A$ following \cite[Section~4]{RT23EUA}. Then one can talk about equations:

\begin{Def}[\cite{LP,RT23EUA}]
	An \textit{equation} between (possibly extended) terms is an expression of the form $$(s=t),$$ where $s$ and $t$ have the same arity. We say that an $\mathbb L$-structure $A$ {\em satisfies} such equation if $s_A=t_A$ in $\V$. \\
	Given a set $\mathbb E$ of equations in $\mathbb L$, we denote by $\Mod(\mathbb E)$ the full subcategory of $\Str(\mathbb L)$ spanned by those $\mathbb L$-structures that satisfy all equations in $\mathbb E$.
\end{Def}

A result that will be useful to characterize $\Mod(\EE)$ relying on its underlying objects and morphisms, and the fact that powers exist:

\begin{lema}\label{ord->enriched}
	Let $\K$ be a $\V$-category with powers and $U\colon \K\to\V$ a $\V$-functor that preserves them. Suppose there is a language $\LL$ and an equational $\LL$-theory $\EE$ together with an isomorphism of ordinary categories $G\colon \K_0\to \Mod(\EE)_0$ making the triangle
	\begin{center}
		\begin{tikzpicture}[baseline=(current  bounding  box.south), scale=2]
			
			\node (a0) at (0,0.8) {$\K_0$};
			\node (b0) at (1.3,0.8) {$\Mod(\EE)_0$};
			\node (c0) at (0.65,0) {$\V$};
			
			\path[font=\scriptsize]
			
			(a0) edge [->] node [above] {$G$} (b0)
			(a0) edge [->] node [left] {$U_0$} (c0)
			(b0) edge [->] node [right] {$(U_\EE)_0$} (c0);
		\end{tikzpicture}	
	\end{center} 
	commute. Then there exists a unique isomorphism of $\V$-categories $E\colon \K\to \Mod(\EE)$ with $E_0=G$ and for which $U_\EE\circ E=U$.
\end{lema}
\begin{proof}
	The $\V$-functor $E$ is defined on objects as $G$; on homs we need to construct maps $E_{AB}\colon\K(A,B)\to \Mod(\EE)(GA,GB)$ for which the following triangle
	 \begin{center}
	 	\begin{tikzpicture}[baseline=(current  bounding  box.south), scale=2]
	 		
	 		\node (a0) at (-0.4,0.8) {$\K(A,B)$};
	 		\node (b0) at (1.7,0.8) {$ \Mod(\EE)(GA,GB)$};
	 		\node (c0) at (0.65,0) {$[UA,UB]$};
	 		
	 		\path[font=\scriptsize]
	 		
	 		(a0) edge [dashed, ->] node [above] {$E_{AB}$} (b0)
	 		(a0) edge [->] node [left] {$U_{AB}\ \ $} (c0)
	 		(b0) edge [->] node [right] {$\ \ (U_\EE)_{AB}$} (c0);
	 	\end{tikzpicture}	
	 \end{center} 
	commutes. To do that, note that for any $X\in\V$ we there is a commutative triangle between the hom-sets of the underlying categories
	\begin{center}
		\begin{tikzpicture}[baseline=(current  bounding  box.south), scale=2]
			
			\node (a0) at (-0.4,0.8) {$\K_0(A,B^X)$};
			\node (b0) at (1.7,0.8) {$ \Mod(\EE)_0(GA,GB^X)$};
			\node (c0) at (0.65,-0.1) {$\V_0(UA,UB^X)$};
			
			\path[font=\scriptsize]
			
			(a0) edge [dashed, ->] node [above] {$G_{AB^X}$} (b0)
			(a0) edge [->] node [left] {$(U_0)_{AB^X}\ \ $} (c0)
			(b0) edge [->] node [right] {$\ \ ((U_\EE)_0)_{AB^X}$} (c0);
		\end{tikzpicture}	
	\end{center} 
	where  now, since $U$ and $U_\EE$ preserves powers, this can be rewritten as 
	\begin{center}
		\begin{tikzpicture}[baseline=(current  bounding  box.south), scale=2]
			
			\node (a0) at (-0.5,0.8) {$\V_0(X,\K(A,B))$};
			\node (b0) at (1.8,0.8) {$\V_0(X,\Mod(\EE)(GA,GB))$};
			\node (c0) at (0.65,-0.1) {$\V_0(X,[UA,UB])$};
			
			\path[font=\scriptsize]
			
			(a0) edge [dashed, ->] node [above] {$G'_{AB}$} (b0)
			(a0) edge [->] node [left] {$\V_0(X,U_{AB})\ \ $} (c0)
			(b0) edge [->] node [right] {$\ \ \V_0(X,(U_\EE)_{AB})$} (c0);
		\end{tikzpicture}	
	\end{center}
	where $G^X_{AB}$ is obtained from $G_{AB^X}$ by composing with the relevant isomorphism. By Yoneda there is then a unique isomorphism $E_{AB}\colon\K(A,B)\to \Mod(\EE)(GA,GB)$ such that $\V_0(X,E_{AB})=G^X_{AB}$. This completes the definition of the isomorphism $E$; the fact that $E$ is itself a $\V$-functor follows from the fact that $U$ and $U_\EE$ are and that $U_\EE$ is faithful.
\end{proof}

\subsection{Languages and theories with respect to a weakly sound class}\label{BackEUA}

We assume $\V$ to be a symmetric monoidal closed and locally presentable category, and we fix a small and dense generator $\A\subseteq\V$ and a locally small weakly sound class $\Phi$.

\begin{obs}
	For the results of the paper it would be enough to assume that $\V_0$ is locally bounded \cite[Section~6]{Kel82:libro} and that has a small dense generator.
	Essentially by definition, every $\V$ which is locally presentable satisfies these two properties. However, whether there are non locally presentable examples (even satisfying just the density hypothesis), is independent from ZFC, and actually equivalent to the Vopenka's principle~\cite[6.B]{AR94:libro}. Thus, we do not really lose much generality by assuming local presentability, and we can also use the results of \cite{RT23EUA}.
\end{obs}

We fix as subcategory of arities $\Phi\I$, the closure of $\I=\{I\}$ under $\Phi$-colimits in $\V$. (The notation is that of the free cocompletion under $\Phi$-colimits).

\begin{prop}
	The $\V$-category $\Phi\I$ is closed in $\V$ under tensor product, so that $\Phi\I$ has $\Phi\I$-copowers. Moreover $$\V\simeq \Phi\I\tx{-Pw}(\Phi\I\op,\V)$$
	induced by the singular $\V$-functor on the inclusion $H\colon \Phi\I\hookrightarrow\V$.
\end{prop}
\begin{proof}
	The first part is done using the definition of $\Phi\I$ as the transfinite union of $\V$-categories $\P_\alpha\I$ with: $\Phi_0\I=\I$, $\Phi_{\alpha+1}\I$ is the full subcategory of $\V$ obtained by adding to $\P_\alpha\I$ every $\Phi$-colimit of a diagram landing in it, we take unions at limit steps.
	
	We prove by recursion that for each $\alpha$ and any $X,Y\in\Phi_{\alpha}\I$, then $X\otimes Y\in\Phi\I$; this will suffice since for any $X,Y\in\Phi\I$ there exists $\alpha$ such that $X,Y\in\Phi_{\alpha}\I$. The statement is trivial for $\alpha=0$. Assume the statement holds for $\alpha$ and consider any $X,Y\in \Phi_{\alpha+1}\I$; we can write $X\cong M*H$ and $Y\cong N*K$ with $M,N\in\Phi$ (or trivial weights, in case one between $X$ or $Y$ actually lies in $\P_\alpha\I$) and $H,K$ landing in $\P_\alpha\I$. Then
	$$X\otimes Y\cong (M*H)\otimes (N*K)\cong M(-)*(N(\square)*(H(-)\otimes K(\square))).$$
	By hypothesis $H(-)\otimes K(\square)$ is in in $\Phi\I$, and hence in some $\P_\beta\I$; thus $X\otimes Y\in\Phi_{\beta+2}\I\subseteq\Phi\I$. The limit step follows from the previous case using again that, if $\alpha$ is limit, then for any $X,Y\in\P_\alpha\I$ there is $\beta<\alpha$ such that $X,Y\in\Phi_{\beta+1}\I$.
	
	For the last part, let $K\colon \Phi\I\hookrightarrow\V$ be the inclusion. Note that $[K,1]\colon \V\to[\Phi\I\op,\V]$ is fully faithful since ($\I$ and hence) $\Phi\I$ is dense in $\V$. Moreover $[K(-),X]$ preserves any limits that $K\op$ preserves, so $\V\subseteq \Phi\I\tx{-Pw}(\Phi\I\op,\V)$. Conversely, if $F\colon\Phi\I\op\to\V$ preserves $\Phi\I$-powers then for any $X\in\Phi\I$ we have
	$$ FX\cong F(X\otimes I)\cong [X,FI]; $$
	showing that $F\cong [K(-),FI]$.	
\end{proof}

Note that $\Phi\I$ is contained in $\V_\Phi$, the full subcategory spanned by the $\Phi$-presentable objects; this is because the unit $I$ is $\Phi$-presentable and $\V_\Phi$ is closed under $\Phi$-colimits, see Section~\ref{soundness}.

\begin{prop}\label{V-Phi-flat}
	The $\V$-category $\V_\Phi$ is the Cauchy completion of $\Phi\I$; moreover left Kan extending along the inclusion $H\colon\Phi\I\hookrightarrow\V$ induces an equivalence
	$$ \tx{Lan}_H\colon [\Phi\I,\V]\xrightarrow{\ \simeq \ } \Phi\+\tx{-Coct}(\V,\V), $$
	where $\Phi\+\tx{-Coct}(\V,\V)$ is the $\V$-category of $\V$-functors $\V\to\V$ preserving $\Phi$-flat colimits.
\end{prop}
\begin{proof}
	That $\V_\Phi$ is the Cauchy completion of $\Phi\I$ follows from Proposition~\ref{K_Phi}(4) since $\Phi\I$ is dense and closed under $\Phi$-colimits in $\V$.
	
	The second part follows from Corollary~\ref{Phi-flat-preserve} since we have an equivalence
	$$ [\Phi\I,\V]\xrightarrow{\ \simeq \ } [\V_\Phi,\V]$$
	induced by left Kan extending along the inclusion $\Phi\I\hookrightarrow\V_\Phi$; indeed, this is the universal property of the Cauchy completion.
\end{proof}

\begin{obs}
	If (the saturation of) $\Phi$ contains the Cauchy weights, then it follows that we already have an equality $\Phi\I=\V_\Phi$. In general this is not the case.
	
	For instance, let $\Phi=\emptyset$ and $\V$ is an $\Ab$-enriched category; then $\Phi\I=\{I\}$, while $\V_\Phi$ contains all the finite products of $I$.
	
\end{obs}

We chose $\Phi\I$, instead of $\V_\Phi$, as class of arities because it provides a smaller and ``simpler'' (since Cauchy colimits are not well understood in general) class of objects of $\V$. Nonetheless, below we could replace $\Phi\I$ with $\V_\Phi$ and everything would still work.

\begin{Def}\label{terms}
	We say that a language $\mathbb L$ is {\em $\Phi$-ary} if every function symbol in $\mathbb L$ has input arity $X\in\Phi\I$.
	An $(X,Y)$-ary $\mathbb L$-term $t$ is called {\em $\Phi$-ary} if $X\in\Phi\I$. Similarly, an $(X,Y)$-ary extended $\mathbb L$-term $t$ is called {\em $\Phi$-ary} if $X\in\Phi\I$.
	An $\mathbb L$-theory $\mathbb T$ is called {\em $\Phi$-ary} if all (extended) terms appearing in it are $\Phi$-ary.
\end{Def}

The following generalizes \cite[5.11]{RT23EUA} from the $\lambda$-ary to the sound case, and is a consequence of \cite[Theorem~5.20]{LP}

\begin{prop}
	Let $\EE$ be a $\Phi$-ary equational theory on a $\Phi$-ary language $\LL$. The $\V$-functor $U \colon \Mod(\EE)\to \V$ is strictly monadic and preserves $\Phi$-flat colimits.
\end{prop}
\begin{proof}
	Given \cite[Example 2.4(5)]{LP}, this follows from \cite[Theorem~5.20]{LP}.
\end{proof}

For other implication we cannot use the results of \cite{LP}, as we want to characterize $\V$-categories of algebras as models of equational theories  which contain only standard (that is, recursively generated) terms. Nonetheless, we can apply the results of \cite{BG,LP23} to $\Phi\I$ since a $\V$-functor $\V\to\V$ is the left Kan extension of its restriction to $\Phi\I$ if and only if it preserves $\Phi$-flat colimits (Proposition~\ref{V-Phi-flat}).

\begin{prop}\label{monad->equations}
	Let $T\colon \V\to\V$ be a monad preserving $\Phi$-flat colimits. Then there exists a $\Phi$-ary equational theory $\EE$, whose equations are defined between standard terms, on a $\Phi$-ary language $\LL$ together with an isomorphism $E\colon \tx{Alg}(T)\to \Mod(\EE)$ making the triangle
	\begin{center}
		\begin{tikzpicture}[baseline=(current  bounding  box.south), scale=2]
			
			\node (a0) at (0,0.8) {$\tx{Alg}(T)$};
			\node (b0) at (1.3,0.8) {$\Mod(\EE)$};
			\node (c0) at (0.65,0) {$\V$};
			
			\path[font=\scriptsize]
			
			(a0) edge [->] node [above] {$E$} (b0)
			(a0) edge [->] node [left] {$U$} (c0)
			(b0) edge [->] node [right] {$U_\EE$} (c0);
		\end{tikzpicture}	
	\end{center} 
	commute.
\end{prop}
\begin{proof}
	The proof is somewhat inspired by that of \cite[5.12]{RT23EUA} but requires some subtle modifications since we cannot use \cite[9.2]{LT20:articolo} because $(\Phi\I)_0$ is not a strong generator of $\V_0$. 
	
	By \cite[2.4]{BG} we can find a $\Phi\I$-theory $H\colon\Phi\I^{\op}\to \T$ for which $\tx{Alg}(T)$ is given by the pullback below.
	\begin{center}
		\begin{tikzpicture}[baseline=(current  bounding  box.south), scale=2]
			
			\node (a0) at (0,0.8) {$\tx{Alg}(T)$};
			\node (a0') at (0.3,0.6) {$\lrcorner$};
			\node (b0) at (1.3,0.8) {$[\T,\V]$};
			\node (c0) at (0,0) {$\V$};
			\node (d0) at (1.3,0) {$[\Phi\I\op,\V]$};
			
			\path[font=\scriptsize]
			
			(a0) edge [right hook->] node [above] {} (b0)
			(a0) edge [->] node [left] {$U$} (c0)
			(b0) edge [->] node [right] {$[H,\V]$} (d0)
			(c0) edge [right hook->] node [below] {$\V(K-,1)$} (d0);
		\end{tikzpicture}	
	\end{center} 
	Such $H\op$ can be chosen to be the left part of the (identity on objects, fully faithful) factorization of 
	$$ \Phi\I\xrightarrow{K}\V\xrightarrow{F} \tx{Alg}(T)$$
	where $F$ is the left adjoint to $U$. In particular then we can assume that $\T$ has and $H$ preserves $\Phi\I$-powers, so that for any $X,Y\in\Phi\I$ the copower of $Y$ by $X$ in $\T$ is simply the (image through $H$ of the) tensor product $X\otimes Y$. 
	Under these assumptions on $H$, a $\V$-functor $G\colon\T\to \V$ preserves $\Phi\I$-powers if and only if $GH$ does. Therefore, since $\V(K-,X)$ preserves these powers for any $X\in\V$, then $\tx{Alg}(T)$ is also defined by the pullback below.
	\begin{center}
		\begin{tikzpicture}[baseline=(current  bounding  box.south), scale=2]
			
			\node (a0) at (0,0.8) {$\tx{Alg}(T)$};
			\node (a0') at (0.3,0.6) {$\lrcorner$};
			\node (b0) at (1.5,0.8) {$\Phi\I\tx{-Pw}[\T,\V]$};
			\node (c0) at (0,0) {$\V$};
			\node (d0) at (1.5,0) {$\Phi\I\tx{-Pw}[\Phi\I\op,\V]$};
			
			\path[font=\scriptsize]
			
			(a0) edge [right hook->] node [above] {} (b0)
			(a0) edge [->] node [left] {$U$} (c0)
			(b0) edge [->] node [right] {$[H,\V]$} (d0)
			(c0) edge [right hook->] node [below] {$\V(K-,1)$} (d0);
		\end{tikzpicture}	
	\end{center} 
	To conclude it is enough to construct a $\Phi$-ary language $\mathbb L$ and a $\Phi$-ary $\mathbb L$-theory $\mathbb E$ for which $\Mod(\mathbb E)$ is also presented as the pullback above.
	
	Consider the $\Phi$-ary language $\mathbb L$ defined by a function symbol $\overline f:(X,Y\otimes Z)$ for any morphism $f\colon Z\to \T(X,Y)$ in $\V$ with $Z\in\A$. In particular, any morphism $g\colon X\to Y$ in $\T$ corresponds to a $(X,Y)$-ary function symbol (taking $Z=I$). Note that for any $g\colon Z\to [X,Y]$, with $X,Y\in\Phi\I$, we have two different $(X,Y\otimes Z)$-ary terms given by $g^t\colon X\to Y\otimes Z$ (from the first rule defining terms) and $\overline{H(g)}$.\\
	The $\mathbb L$-theory $\mathbb E$ is given by the following equations:\begin{enumerate}
		\item[(a)] $ k^Y(\overline f)=\overline{fk}$, for any $k\colon Z'\to Z$ in $\A$ and $f\colon Z\to \T(X,Y)$;
		\item[(b)] $\overline f_2^{Z_1}(\overline f_1)= \overline{M(f_2\otimes f_1)} $, for any $f_1\colon Z_1\to \T(X_1,X_2)$ and $f_2\colon Z_2\to \T(X_2,X_3)$;
		\item[(c)] $g^t=\overline{H(g)}$, for any $g\colon Z\to [X,Y]$ with $X,Y\in\Phi\I$;
		\item[(d)] $\overline{Z\otimes f} = \overline f^Z$ for any morphism $f\colon X\to Y$ in $\T$ and $Z\in\Phi\I$.
	\end{enumerate}
	Now, to give an $\mathbb L$-structure is the same as giving an object $A$ of $\V$ together with a map 
	$$ \overline f_A\colon A^X\to A^{Y\otimes Z} $$
	for any $f\colon Z\to \T(X,Y)$ in $\V$ with $Z\in\A$. Taking the transpose of $\overline f_A$ with respect to $Z$ this defines a function
	$$ \V_0(Z,\T(X,Y))\longrightarrow \V_0(Z,[A^X,A^Y]) $$
	for any $Z\in\A$. Then, $A$ satisfies the equations in (a) if and only if the map above is natural in $Z\in\A$, if and only if (by density of $\A$) it is induced by homming out of a morphism 
	$$ F_{X,Y}\colon \T(X,Y)\longrightarrow [A^X,A^Y] $$
	in $\V$ for any $X,Y\in\Phi\I$. 
	
	Using that $\A$ is a strong generator of $\V$, the equations in (b) say that the morphisms $F_{X,Y}$ respect composition, meaning that the square below
	\begin{center}
		\begin{tikzpicture}[baseline=(current  bounding  box.south), scale=2]
			
			\node (a0) at (0,0.9) {$\T(X_2,X_3)\otimes \T(X_1,X_2)$};
			\node (b0) at (2.3,0.9) {$\T(X_1,X_3)$};
			\node (c0) at (0,0) {$[A^{X_2},A^{X_3}]\otimes [A^{X_1},A^{X_2}]$};
			\node (d0) at (2.3,0) {$[A^{X_1},A^{X_3}]$};
			
			\path[font=\scriptsize]
			
			(a0) edge [->] node [above] {$M$} (b0)
			(a0) edge [->] node [left] {$F_{X_2,X_3}\otimes F_{X_1,X_2}$} (c0)
			(b0) edge [->] node [right] {$F_{X_1,X_3}$} (d0)
			(c0) edge [->] node [below] {$M$} (d0);
		\end{tikzpicture}	
	\end{center}
	commutes for any $X_1,X_2,X_3\in\Phi\I$. 
	
	The axioms in (c) say that the following triangle commutes
	\begin{center}
		\begin{tikzpicture}[baseline=(current  bounding  box.south), scale=2]
			
			\node (a0) at (0,0.9) {$[X,Y]$};
			\node (c0) at (0,0) {$\T(X,Y)$};
			\node (d0) at (2,0) {$[A^{X},A^{Y}]$};
			
			\path[font=\scriptsize]
			
			(a0) edge [->] node [above] {$\ \ \ A^{(-)}_{X,Y}$} (d0)
			(a0) edge [->] node [left] {$H_{X,Y}$} (c0)
			(c0) edge [->] node [above] {$F_{X,Y}$} (d0);
		\end{tikzpicture}	
	\end{center}
	for any $X,Y\in\Phi\I$. This implies in particular that $F_{X,X}$ preserves the identity.
	
	It follows that to give an $\LL$-structure satisfying (a)-(c) is the same as giving an object $A\in\V$ together with a $\V$-functor $F\colon\T\to\V$ such that $FH=(-)^X$. In addition, the $\LL$-structure satisfies condition (d) if and only the corresponding $\V$-functor preserves $\Phi\I$-powers.
	
	As a consequence $\Mod(\mathbb E)_0$ is a pullback
	\begin{center}
		\begin{tikzpicture}[baseline=(current  bounding  box.south), scale=2]
			
			\node (a0) at (0,0.8) {$\Mod(\mathbb E)_0$};
			\node (a0') at (0.3,0.6) {$\lrcorner$};
			\node (b0) at (1.55,0.8) {$\Phi\I\tx{-Pw}[\T,\V]_0$};
			\node (c0) at (0,0) {$\V$};
			\node (d0) at (1.55,0) {$\Phi\I\tx{-Pw}[\Phi\I\op,\V]_0$};
			
			\path[font=\scriptsize]
			
			(a0) edge [right hook->] node [above] {} (b0)
			(a0) edge [->] node [left] {$(U_\EE)_0$} (c0)
			(b0) edge [->] node [right] {$[H,\V]_0$} (d0)
			(c0) edge [right hook->] node [below] {$\V(K-,1)_0$} (d0);
		\end{tikzpicture}	
	\end{center} 
	of ordinary categories. It follows that the hypotheses of Lemma~\ref{ord->enriched} are satisfied, and hence we have an equivalence $E\colon \tx{Alg}(T)\to \Mod(\EE)$ as requested.
\end{proof}

\subsection{Characterization theorem} $ $

In this section we use the results above to prove the main characterization theorem for $\V$-categories of models of equational $\Phi$-theories. Recall from \cite{RT23EUA} that an object $G$ of a $\V$-category $\K$ is called {\em $\V$-projective} if $\K(G,-)$ preserves coequalizers of $\K(G,-)$-split pairs. Note that $(1)\Leftrightarrow(3)$ was shown as \cite[Theorem~5.26]{LP}; the improvement here is that it is enough to use the recursively generated terms.

\begin{teo}\label{char-single}
	The following are equivalent for a $\V$-category $\K$: \begin{enumerate}\setlength\itemsep{0.07em}
		\item $\K\simeq\Mod(\mathbb E)$ for a $\Phi$-ary equational theory $\mathbb E$ involving extended terms;
		\item $\K\simeq\Mod(\mathbb E)$ for a $\Phi$-ary equational theory $\mathbb E$ involving only standard terms;
		\item $\K\simeq\tx{Alg}(T)$ for a monad $T$ on $\V$ preserving $\Phi$-flat colimits;
		\item $\K$ is cocomplete and has a $\Phi$-presentable and $\V$-projective strong generator $G\in\K$;
		\item $\K\simeq\Phi\I\tx{-Pw}(\T,\V)$ is equivalent to the $\V$-category of $\V$-functors preserving $\Phi\I$-powers, for some $\Phi\I$-theory $\T$.
	\end{enumerate}
	In particular $\K$ is locally $\Phi$-presentable.
\end{teo}
\begin{proof}
	$(1)\Rightarrow(4)$. By monadicity, the forgetful $\V$-functor $U_\mathbb E\colon \K\simeq\Mod(\mathbb E)\to \V$ is a right adjoint. Thus the value of its left adjoint $L$ at $I$ gives an object $G:=LI\in\K$ for which $U_\mathbb E\cong\K(G,-)$. Since $U_\mathbb E$ is conservative, preserves $\Phi$-flat colimits and $U_\mathbb E$-split coequalizers (being monadic), it follows that $G$ has the desired properties. This also makes $\K$ locally $\Phi$-presentable.
	
	$(4)\Rightarrow(3)$. Note that the $\V$-category $\K$ is locally $\Phi$-presentable and that $$U_\K:=\K(G,-)\colon\K\to\V$$ is (by hypothesis) continuous and preserves $\Phi$-flat colimits as well as coequalizers of $U$-split pairs. Thus $U_\K$ has a left adjoint (Proposition~\ref{lpAFT}) and is $\Phi$-ary monadic by the monadicity theorem.
	
	$(3)\Rightarrow(2)$ is given by Proposition~\ref{monad->equations} and $(2)\Rightarrow(1)$ is trivial.

	$(5)\Leftrightarrow(3)$. By \cite[Theorem~19]{BG}, a $\V$-category $\K$ is equivalent to the $\V$-category of models of a $\Phi$-ary monad $T$ if and only if there is a $\Phi\I$-theory $H\colon\Phi\I\op\to \T$ for which $\K$ is given by the bipullback below.
	\begin{center}
		\begin{tikzpicture}[baseline=(current  bounding  box.south), scale=2]
			
			\node (a0) at (0,0.8) {$\K$};
			\node (a0') at (0.3,0.6) {$\lrcorner$};
			\node (b0) at (1.3,0.8) {$[\T,\V]$};
			\node (c0) at (0,0) {$\V$};
			\node (d0) at (1.3,0) {$[\Phi\I\op,\V]$};
			
			\path[font=\scriptsize]
			
			(a0) edge [right hook->] node [above] {} (b0)
			(a0) edge [->] node [left] {$U$} (c0)
			(b0) edge [->] node [right] {$[H,\V]$} (d0)
			(c0) edge [right hook->] node [below] {$\V(K-,1)$} (d0);
		\end{tikzpicture}	
	\end{center} 
But we know that $\V\simeq \Phi\I\tx{-Pw}(\Phi\I\op,\V)$; thus, since $H$ is bijective on objects and preserves $\Phi\I$-powers, it follows that $\K$ is a bipullback as above if and only if $\K\simeq\Phi\I\tx{-Pw}(\T,\V)$.
\end{proof}

\begin{obs}
	Let $J\colon\Phi\I\hookrightarrow\V$ be the inclusion. It follows from \cite[7.2.9 \& 7.3.7]{arkor2022monadic} and \cite[7.8]{AM2024pullback} that the conditions of Theorem~\ref{char-single} above are further equivalent to: $\K\simeq\tx{Alg}(T)$ for a $J$-relative monad $T$.
\end{obs}

As done in \cite{RT23EUA}, the subcategories of $\Str(\LL)$ defined by $\Phi$-ary equational theories can be characterized as certain orthogonality classes.

\begin{prop}\label{regular-epi-orth}
	Let $\mathbb L$ be a $\Phi$-ary language. Then classes defined by $\Phi$-ary equational $\mathbb L$-theories in $\Str(\mathbb L)$ are precisely given by orthogonality classes defined with respect to maps of the form 
	$$ h\colon FX\twoheadrightarrow W $$
	in $\Str(\mathbb L)$, where $X\in\Phi\I$ and $h$ is a regular epimorphism.
\end{prop}
\begin{proof}
	Same argument used to prove \cite[Proposition~5.19]{RT23EUA}, with the only difference that now $X\in\Phi\I$.
\end{proof}

\subsection{Examples}

\subsubsection{Finite products}\label{fp}

We assume the class $\bo{Fp}$ for finite products to be weakly sound; by \cite{KL93FfKe:articolo} this is the case whenever $\V$ is endowed with the cartesian closed structure. This in some sense generalizes the single-sorted case of~\cite{Par23} since we do not assume that the unit of $\V_0$ is a generator.

We say that a weight $M$ is {\em $\V$-sifted} if it is $\bo{Fp}$-flat; that is, if $M$-colimits commute in $\V$ with finite products. A $\V$-functor $T\colon\V\to\V$ preserves $\V$-sifted colimits if and only if it is the left Kan extension of its restriction to $\bo{Fp}(\I)$, the full subcategory of $\V$ spanned by the finite coproducts of the terminal object (Proposition~\ref{V-Phi-flat}). These $\V$-functors are often called {\em strongly finitary}.

The corresponding notions of $\bo{Fp}$-ary languages and theories are particularly simple since they only involve function symbols and terms of the form $t:(n,Y)$, where $n$ is a natural number and $Y\in\V$. The interpretation of such a term $t$ on an $\LL$-structure $A$ is a map 
$$ t_A\colon A^n\to A^Y$$
where $A^n$ is the product of $A$ with itself $n$-times (so, no purely enriched notion of powers involved in the domain). Theorem~\ref{char-single} in this context becomes:

\begin{teo}\label{char-discrete}
	The following are equivalent for a $\V$-category $\K$: \begin{enumerate}
		\item $\K\simeq\Mod(\mathbb E)$ for a $\bo{Fp}$-ary equational theory $\mathbb E$ involving extended terms;
		\item $\K\simeq\Mod(\mathbb E)$ for a $\bo{Fp}$-ary equational theory $\mathbb E$ involving standard terms;
		\item $\K\simeq\tx{Alg}(T)$ for a monad $T$ on $\V$ preserving $\V$-sifted colimits;
		\item $\K$ is cocomplete and has a $\bo{Fp}$-presentable and $\V$-projective strong generator $G\in\K$;
		\item $\K\simeq\tx{Fp}(\T,\V)$ is equivalent to the $\V$-category of $\V$-functors preserving finite products (or equivalently finite discrete powers), for some $\bo{Fp}(\I)$-theory $\T$.
	\end{enumerate}
	Moreover, they always imply:\begin{enumerate}
		\item[(6)] $\K$ is cocomplete and has a $\bo{Fp}$-presentable strong generator $G\in\K$.
	\end{enumerate}
	While, (6) implies (1)-(5) if reflexive coequalizers are $\V$-sifted.
\end{teo}
\begin{proof}
	Mostly follows from Theorem~\ref{char-single}. It is enough to prove $(6)\Rightarrow(3)$ with the assumption that reflexive coequalizers are $\V$-sifted. By hypothesis $\K(G,-)$ is a conservative right adjoint that preserves sifted colimits, and hence also coequalizers of reflexive pairs (since these are $\V$-sifted). Thus $\K(G,-)$ is monadic by the crude monadicity theorem, so that (3) holds.
\end{proof}

Reflexive coequalizers are $\V$-sifted whenever $\V$ is a finitary variety (that is, a locally strongly finitely presentable ordinary category). By \cite[Proposition~3.6]{ADV2023sifted} and the comments below that, reflexive coequalizers (as well as all conical sifted colimits)are $\V$-sifted whenever $\V$ is cartesian closed. This includes in particular the case $\V=\bo{Cat}$ (see also Appendix~\ref{Cat-sifted}) which is not a variety.

\subsubsection{Finite products plus projective powers}\label{fpp}

Let $\bo{Fp}$ be the ordinary class for finite products; in this section we fix a base of enrichment $\V$ for which $\V_0$ is locally $\bo{Fp}$-presentable and the $\bo{Fp}$-presentable objects are closed under the tensor product (but need not contain the unit). Locally $\bo{Fp}$-presentable categories are commonly called {\em strongly finitely presentable}~\cite{LR11NotionsOL}. For such a $\V$, an object $X$ is $\bo{Fp}$-presentable (by definition) if the functor $\V_0(X,-)$ preserves sifted colimits; it is well-known that this is equivalent to preservation of filtered colimits and reflexive coequalizers~\cite{ARV2010sifted}. When $\V_0$ is moreover exact, preservation of reflexive equalizers above is equivalent to preservation of regular epimorphisms. These bases have been considered in ~\cite{LR11NotionsOL,LT20:articolo}, see Example~\ref{inducedfromSet}.

We consider the class of weights $\bo{Fpp}:=\Phi_\bo{Fp}$ induced by taking finite products and powers by $\bo{Fp}$-presentable objects. By Proposition~\ref{Phi-V-sound} this is weakly sound, and hence we can apply the results of Section~\ref{UAsection}. Moreover, by Corollary~\ref{DvsPhiD}, a local $\bo{Fpp}$-presentability can be tested using conical sifted colimits; this will be useful in the applications of Theorem~\ref{char-projective} below.

The main difference with respect to the setting of Section~\ref{fp} above (where we consider only finite products, but no powers) is that for $\V$ an additive category, the class of finite products involves non interesting monads (those preserving all colimits), while the class of finite products plus $\bo{Fp}$-presentable powers induces monads preserving conical sifted colimits.

\begin{teo}\label{char-projective}
	The following are equivalent for a $\V$-category $\K$: \begin{enumerate}
		\item $\K\simeq\Mod(\mathbb E)$ for a projective equational theory $\mathbb E$ involving extended terms;
		\item $\K\simeq\Mod(\mathbb E)$ for a projective equational theory $\mathbb E$ involving standard terms;
		\item $\K\simeq\tx{Alg}(T)$ for a monad $T$ on $\V$ preserving (conical) sifted colimits;
		\item $\K$ is cocomplete and has a strongly finitely presentable strong generator $G\in\K$;
		\item $\K\simeq\tx{Fpp}(\T,\V)$ is equivalent to the $\V$-category of $\V$-functors preserving finite products and powers by $\bo{Fp}$-presentable objects, for some $\bo{Fpp}$-theory $\T$.
	\end{enumerate}
\end{teo}
\begin{proof}
	The general theorem gives that (1),(2),(3), and (5) are all equivalent to:\begin{enumerate}
		\item[(4')] $\K$ is cocomplete and has a strongly finitely presentable and $\V$-projective strong generator $G\in\K$;
	\end{enumerate}
	which clearly implies (4). For the converse we argue as in the finite product case: if we assume (4), then $\K(G,-)$ is a conservative right adjoint that preserves sifted colimits, and hence also coequalizers of reflexive pairs (since these are sifted). Thus $\K(G,-)$ is monadic by the crude monadicity theorem, so that (3) holds.
\end{proof}

\appendix
\section{Sifted colimits for $\V=\bo{Cat}$}\label{Cat-sifted}

In this section we focus on the cartesian closed category $\bo{Cat}$ of small categories and consider the sound class of (conical) weights $\bo{Fp}$ for finite products. We study locally $\bo{Fp}$-presentable 2-categories and characterize when 2-functors out of those preserve $\bo{Fp}$-flat (from now on called 2-sifted) colimits.

\begin{Def}
	Say that a weight $M\colon\C\op\to\bo{Cat}$ is {\em 2-sifted} if it is $\bo{Fp}$-flat; that is, if $M$-weighted colimits commute in $\bo{Cat}$ with finite products; we denote by $\bo{2}\tx{-}\bo{Sind}$ the class of 2-sifted weights.
\end{Def}

Recall~(\cite{ARValgebraic}) that an ordinary category $\B$ is called sifted if $\B$-colimits commute in $\Set$ with finite products. In the 2-categorical context, we call {\em conical sifted colimits} those conical colimits taken along a sifted category. 

Another notion of 2-categorical colimit, that will be essential in the characterization of 2-sifted colimits, is that of {\em reflexive codescent object}. The definition can be found for instance in \cite[Section~2.2]{Bou10:PhD} where they are called ``codescent objects of strict reflexive coherent data". We do not really need to recall the definition here since our results will be quite formal in nature, and we rely on the results proved in \cite{Bou10:PhD} for the more technical parts (that would require an explicit definition).

Conical sifted colimits and reflexive codescent objects are 2-sifted, see \cite{lack2002codescent} and \cite[Section~8.4]{Bou10:PhD}. The proof that these colimits commute in $\bo{Cat}$ with finite products is the actual technical bit; everything else we do here follows by abstract arguments.

The first thing we observe is that, under the existence of finite coproducts, 2-sifted colimits are generated by conical sifted colimits and reflexive codescent objects.

Let $\bo{Ref}$ be the class of weights for reflexive coequalizers and reflexive codescent objects. And recall that a 2-category is finitely cocomplete, in the (enriched) 2-categorical sense, if it has finite conical limits and powers by $\bb{2}$ (or by any finitely presentable category).

\begin{lema}\label{ref+cop}
	A 2-category $\C$ is finitely cocomplete if and only if it has finite coproducts and $\bo{Ref}$-colimits. A 2-functor from such a $\C$ preserves finite colimits if and only if it preserves coproducts and $\bo{Ref}$-colimits.
\end{lema}
\begin{proof}
	One direction is clear. Assume now that $\C$ has finite coproducts and $\bo{Ref}$-colimits; then it has coequalizers since these are generated by coproducts and reflexive coequalizers. Regarding finite powers, if $X$ is a finite category then we can see it as the reflexive codescent object of the finite discrete categories $\{X_i\}_{i=0,1,2}$ given by the truncated nerve of $X$. It follows that, given $C\in\C$, the copower $X\cdot C$ in $\C$ can be seen as the reflexive codescent object of $\{X_i\cdot C\}_{i=0,1,2}$, and these exist in $\C$ because they are coproducts of copies of $\C$ (the $X_i$'s being discrete categories). As a consequence $\C$ is finitely cocomplete. The fact about preservation is a direct consequence of this.
\end{proof}

When $\C=\bo{Cat}$ the following is a consequence of \cite[Theorem~8.31]{Bou10:PhD}.

\begin{cor}
	Let $\C$ be a 2-category with finite coproducts; the following are equivalent:\begin{enumerate}
		\item $\C$ is cocomplete;
		\item $\C$ has 2-sifted colimits;
		\item $\C$ has sifted colimits and reflexive codescent objects;
		\item $\C$ has filtered colimits, reflexive coequalizers, and reflexive codescent objects.
	\end{enumerate}
\end{cor}
\begin{proof}
	$(1)\Rightarrow(2)\Rightarrow(3)\Rightarrow(4)$ are trivial. For $(4)\Rightarrow(1)$, by Lemma~\ref{ref+cop} the 2-category $\C$ has all finite colimits, since it has also filtered colimits, it follows that it is cocomplete.
\end{proof}

Next we are interested in characterizing preservation of 2-sifted colimits out of a locally $\bo{Fp}$-presentable 2-category.

\begin{prop}\label{Ref-complete}
	Let $\C$ be a 2-category with finite coproducts; then the free cocompletion $\tx{Ref}(\C)$ is finitely cocomplete and the inclusion $J\colon\C\hookrightarrow\tx{Ref}(\C)$ preserves finite coproducts. Moreover, if $\A$ is finitely cocomplete, a 2-functor $M\colon\C\to\A$ preserves finite coproducts if and only if $\tx{Lan}_JM$ is finitely cocontinuous.
\end{prop}
\begin{proof}
	It is easy to see that $\tx{Ref}(\C)$ can be defined as follows: consider $\D_0=\C$, and for $k\geq0$ define $\D_{2k+1}$ as the full subcategory of $[\C\op,\bo{Cat}]$ spanned by the reflexive coequalizers of elements of $\D_{2k}$, and $\D_{2k+2}$ as the full subcategory of $[\C\op,\bo{Cat}]$ spanned by the reflexive codescent objects of elements of $\D_{2k+1}$. Then $\tx{Ref}(\C)=\cup_{n}\D_n$.
	
	By Lemma~\ref{ref+cop} above, to prove that $\tx{Ref}(\C)$ is finitely cocomplete it is enough to show that it has finite coproducts. Since $\bo{Ref}$-colimits are 2-sifted it follows that $\tx{Ref}(\C)\subseteq \tx{2}\tx{-}\tx{Sind}(\C)$; thus it is enough to show that $\tx{Ref}(\C)$ is closed in $\tx{2}\tx{-}\tx{Sind}(\C)$ under finite coproducts (which exist in $\tx{2}\tx{-}\tx{Sind}(\C)$ being cocomplete). We do that by showing that $\D_n$ is closed under finite coproducts in $\tx{2}\tx{-}\tx{Sind}(\C)$ for each $n\in\mathbb{N}$.
	
	For $n=0$ the result is true simply because $\D_0=\C$ has finite coproducts and 2-sifted colimits commute in $\bo{Cat}$ with finite products. Consider now $n=2k+1$ and assume by inductive hypothesis that $\D_{2k}$ is closed under finite coproducts in $\tx{2}\tx{-}\tx{Sind}(\C)$. Given $X,Y\in\D_{2k+1}$, we can write $X$ and $Y$ as reflexive coequalizers of maps in $\D_{2k}$; take the coproduct of these maps in $\D_{2k}$, then their coequalizer $Z$ in $\tx{2}\tx{-}\tx{Sind}(\C)$ is the coproduct of $X$ and $Y$ (by the commutativity of colimits with other colimits and the fact that $\D_n$ is closed under finite coproducts). Since reflexive coequalizers in $\D_{2k+1}$ are computed as in $\tx{2}\tx{-}\tx{Sind}(\C)$, it follows that $Z$ actually lies in $\D_{2k+1}$ and is the coproduct of $X$ and $Y$. The same argument applies for $\D_{2k+2}$. 
	
	As a consequence $\tx{Ref}(\C)$ is finitely cocomplete and the inclusion $J\colon\C\hookrightarrow\tx{Ref}(\C)$ preserves finite coproducts. Moreover, by the explicit construction of coproducts given above, it follows that if a 2-functor $M\colon\C\to\A$ preserves finite coproducts then $\tx{Lan}_JM$ preserves them too. The last part of the statement is then a consequence of Lemma~\ref{ref+cop} and the universal property of $\bo{Ref}$-cocompletions.
\end{proof}

A consequence is:

\begin{cor}\label{fp<fin}
	Let $\C$ be any 2-category; then $\tx{Fin}(\C)\simeq\tx{Ref}(\tx{Fp}(\C))$.
\end{cor}
\begin{proof}
	By Proposition~\ref{Ref-complete} above, the composite 
	$$ Z\colon\C\hookrightarrow \tx{Fp}(\C) \hookrightarrow \tx{Ref}(\tx{Fp}(\C)), $$
	of the two inclusions, satisfies the universal property of free cocompletion under finite colimits.
\end{proof}

Below we denote by $\tx{Ind}(-)$ the free cocompletion under filtered colimits.

\begin{cor}\label{2-sifted}
	Let $\C$ be a 2-category with finite coproducts; then $$\tx{2}\tx{-}\tx{Sind}(\C)\simeq\tx{Ind}(\tx{Ref}(\C)).$$%
	The equivalence being induced by the map $\tx{Ind}(\tx{Ref}(\C))\to \tx{2-Sind}(\C)$ obtained by left Kan extending the inclusion $\C\hookrightarrow \tx{2-Sind}(\C)$.
\end{cor}
\begin{proof}
	We have the following chain of equivalences
	\begin{equation*}
		\begin{split}
			\tx{2-Sind}(\C)&\simeq \tx{Fp}(\C\op,\bo{Cat})\\
			&\simeq \tx{Lex}(\tx{Ref}(\C)\op,\bo{Cat})\\
			&\simeq \tx{Ind}(\tx{Ref}(\C))\\
		\end{split}
	\end{equation*}
	where the first equivalence is given by soundness of the class of finite products, the second is obtained by right Kan extending along $J\op\colon\C\op\hookrightarrow\tx{Ref}(\C)\op$ and is a consequence of (the dual of) Proposition~\ref{Ref-complete}. The third equivalence is given by soundness of the class of finite limits, plus the fact that freely adding filtered colimits is the same as adding $\bo{Lex}$-flat colimits (by \cite[Theorem~3.13]{LT21:articolo} applied to $\V=\bo{Cat}$).
\end{proof}

When $\K=\L=\bo{Cat}$ the result below was proved in \cite[Corollary~8.45]{Bou10:PhD}. 

\begin{teo}
	Let $\K$ be a cocomplete 2-category and $F\colon\K\to\L$ be a 2-functor into a 2-category $\L$ with 2-sifted colimits. Then $F$ preserves 2-sifted colimits if and only if it preserves filtered colimits, reflexive coequalizers, and reflexive codescent objects.
\end{teo}
\begin{proof}
	One implication is clear. Suppose then that $F$ preserves filtered colimits, reflexive coequalizers, and reflexive codescent objects. 
	
	Let $\K$ be cocomplete and consider a sifted weight $M\colon\C\op\to\bo{Cat}$ together with a diagram $H\colon\C\to\K$. Without loss of generality we can assume that $\C$ has finite coproducts; indeed, we can replace $\C$ with its free cocompletion $\tx{Fp}(\C)$ under finite coproducts, $M$ with $M':=\tx{Lan}_{J\op}M$ and $H$ with $H':=\tx{Lan}_JH$, where $J\colon\C\to\tx{Fp}(\C)$ is the inclusion (then $M'$ is still 2-sifted and $M*H\cong M'*H'$).
	
	Now, it follows that $M\in \tx{2-Sind}(\C)$, which by Corollary~\ref{2-sifted} above is equivalent to $\tx{Ind}(\tx{Ref}(\C))$. Therefore $M$ is a filtered colimit of $\bo{Ref}$-colimits of representables. Thus, the preservation of $M$-weighted colimits follows from that of filtered colimits and $\bo{Ref}$-colimits, as claimed.
	
\end{proof}


\end{document}